\documentclass{amsart}

\usepackage{hyperref}

\usepackage{xcolor}

\usepackage{amssymb}
\usepackage{amsthm}
\usepackage{amstext}
\usepackage{amscd}
\usepackage{amsfonts} 					
\usepackage{amsmath}

\usepackage{mathtools}

\usepackage{verbatim} 

\usepackage{enumitem} 
\numberwithin{equation}{section}

\usepackage{tikz}						
\usetikzlibrary{shapes.geometric, arrows,matrix}

\newtheorem{thm}{Theorem}[section]

\newtheorem{lemma}[thm]{Lemma}
\newtheorem{prop}[thm]{Proposition}
\newtheorem{cor}[thm]{Corollary}

\newtheorem{conj}[thm]{Conjecture}
\newtheorem{claim}[thm]{Claim}
\newtheorem{ex}[thm]{Example}
\newtheorem{rmk}[thm]{Remark}

\newtheorem{THM}{Theorem}

\newtheorem{COR}[THM]{Corollary}
\newtheorem*{CONJ}{Conjecture}

\newcommand{\mb}{\mathbb}

\newcommand{\C}{\mb C}
\newcommand{\Pj}{\mb P}

\newcommand{\mf}{\mathfrak}

\newcommand{\m}{\mf m}

\newcommand{\mc}{\mathcal}

\newcommand{\F}{\mc F}

\newcommand{\mcE}{\mc E}
\newcommand{\mcL}{\mc L}
\newcommand{\Lie}{\mc{L}}

\newcommand{\mr}{\mathrm}

\newcommand{\hatnabla}{\widehat{\nabla}}
\newcommand{\de}{\partial}

\newcommand\restr[2]{{
  \left.\kern-\nulldelimiterspace 
  #1 
  \vphantom{\big|} 
  \right|_{#2} 
  }}

\DeclareMathOperator{\mrrestr}{restr}

\newcommand{\OmegaDE}{\Omega_X^1(D)(\mcE)}

\newcommand{\DX}[1]{\mc{D}\mathit{iff}_X^{\le #1}}
\newcommand{\DXF}[1]{\mc{D}\mathit{iff}_{X/\F}^{\le #1}}
\newcommand{\PX}{\mc{P}_X}
\newcommand{\PXF}{\mc{P}_{X/\F}}
\newcommand{\PPjnF}{\mc{P}_{\Pj^n/\F}}
\newcommand{\PPjtwoF}{\mc{P}_{\Pj^2/\F}}

\newcommand{\jetX}{\mc{J}_X}
\newcommand{\jetXF}{\mc{J}_{X/\F}}

\newcommand{\germe}{\mathcal{O}}
\newcommand{\mcHom}{\mathcal{H}om}

\DeclareMathOperator{\codim}{codim}

\DeclareMathOperator{\Sym}{Sym}

\DeclareMathOperator{\Hom}{Hom}

\DeclareMathOperator{\Aff}{Aff}

\newcommand{\aff}{\mf{aff}}

\DeclareMathOperator{\Aut}{Aut}

\newcommand{\ideal}[1]{\left< #1\right>}

\newcommand{\CNF}{{N^*_\F}}
\newcommand{\CTF}{\Omega_\F}

\DeclareMathOperator{\sing}{sing}
\newcommand{\idealsing}{I_{\mathrm{sing}}}
\newcommand{\idealpers}{I_{\mathrm{pers}}}

\newcommand{\df}{d_{\mc F}}

\title{Jets of flat partial connections II}
\author[G. Fazoli]{Gabriel Fazoli}
\address{Univ Rennes, CNRS, IRMAR-UMR 6625, F-35000 Rennes, France}
\email{gabrielfazoli@gmail.com}
\date{}

\begin{document}

\begin{abstract}
    We define and study jets of flat partial connections with respect to singular foliations. In particular, we use the first sheaf of transverse jets to address the problem of extending a flat partial connection to a (flat) meromorphic connection. We then apply our results on the extension problem to classical problems in foliation theory, such as the existence of transversely affine structures and invariant hypersurfaces.
\end{abstract}

\maketitle

\section{Introduction}

This work is a sequel to \cite{fazoli25-arXiv:2505.11662}, but it has been written in a manner that it can be read independently.

\subsection{Foliations, partial connections and jets of flat partial connections}\label{Subsection: introduction, partial connections, jets}

Let $X$ be a smooth complex variety, either in the analytic or in the algebraic setting. Recall that a \emph{foliation} $\F$ on $X$ is a coherent involutive subsheaf $T_{\F} \subset T_X$ such that the quotient $T_X/T_{\F}$ is torsion free. We refer to $T_{\F}$ as the \emph{tangent sheaf} of the foliation, and we define $\rm{rk} (T_{\F})$ to be the \emph{dimension} of $\F$.

Let $\mcE$ be a coherent sheaf over $X$. We define
\[
    \CTF^1(\mcE):= \mcHom_{\germe_X}(T_{\F},\mcE)
\]
to be the sheaf of \emph{foliated 1-forms with coefficients in $\mcE$}. When $\mcE = \germe_X$, we denote $\CTF^1(\germe_X)$ simply by $\CTF^1$, and we refer to it as the \emph{cotangent sheaf} of $\F$. A $\F$-partial connection on $\mcE$ is a $\C$-linear map
\begin{equation*}
    \begin{split}
        \nabla: \mcE & \rightarrow \CTF^1(\mcE) \\
        s & \mapsto \big(v\in T_{\F} \mapsto \nabla_v(s) \big)
    \end{split}
\end{equation*}
such that, for all $f\in \germe_X, s\in \mcE, v\in T_{\F}$, the following Leibniz rule holds:
\begin{equation}\label{Eq: Leibniz rule}
    \nabla_v(f\cdot s) = v(f) \cdot s + f \cdot \nabla_v(s).
\end{equation}
When $\F$ is clear from the context, we will simply refer to $\nabla$ as a partial connection. Additionally, we say that $\nabla$ is \emph{flat} if
\begin{equation}\label{Eq: flatness}
    \nabla_{[v,w]} = \nabla_v \circ \nabla_w - \nabla_w\circ \nabla_v, \forall v,w\in T_{\F}.
\end{equation}

We refer to \cite[Section 3]{fazoli25-arXiv:2505.11662} for a more detailed discussion on partial connection, including examples and key properties.

Both in \cite[Definition 3.1]{biswas-02-zbMATH01786109} and \cite[Section 4.2]{fazoli25-arXiv:2505.11662}, the notion of \emph{jets of flat partial connections} is defined in the context of analytic varieties, smooth foliations, and flat partial connections on locally free sheaves. Following \cite{fazoli25-arXiv:2505.11662}, the construction proceeds as follows. 

Given a flat partial connection $\nabla$ on a locally free sheaf $\mcE$, we define the \emph{$k$-th sheaf of transverse jets of $(\mcE,\nabla)$} as the locally free $\germe_X$-module generated by the $k$-jets of flat sections of $\nabla$, namely: 
\[
\jetXF^k(\nabla):= \left\{ \sum f_i \cdot d^k(s_i) \mid f_i \in \germe_X, s_i \in \ker \nabla \right\} \subset \jetX^k(\mcE).
\]
Next, denoting by $\germe_{X/\F}$ the ring of first integrals of $\F$, we consider the locally free $\germe_{X/\F}$-module generated by the $k$-jets of flat sections of $\nabla$, that is, the sheaf
\[
\jetXF^k(\ker \nabla) := \left\{ \sum f_i \cdot d^k(s_i) \mid f_i \in \germe_{X/\F}, s_i \in \ker \nabla \right\} \subset \jetX^k(\mcE).
\]
Since
\begin{enumerate}[label = -]
    \item the rank of $\jetXF^k(\ker \nabla)$ as an $\germe_{X/\F}$-module is equal the rank of $\jetXF^k(\nabla)$ (see \cite[Proposition 4.4]{fazoli25-arXiv:2505.11662}), and
    \item $\jetXF^k(\ker \nabla)$ generates $\jetXF^k(\nabla)$ as a $\germe_X$-module,
\end{enumerate}
there exists a unique flat partial connection $\nabla^k$ on $\jetXF^k(\nabla)$ such that $\ker \nabla^k = \jetXF^k(\ker \nabla)$ (see \cite[Corollary 3.8 and Corollary 4.5]{fazoli25-arXiv:2505.11662}). The pair $(\jetXF^k(\nabla),\nabla^k)$ is called the \emph{$k$-th jet of the flat partial connection $(\mcE,\nabla)$}.

These objects have been employed to study transverse structures to foliations. See, for instance, \cite[Theorem 5.3]{biswas-02-zbMATH01786109} and \cite[Corollary 5.5 and Theorem 5.9]{fazoli25-arXiv:2505.11662}. This is a very natural perspective, as the classical concept of the jet of sheaf is also used in the study of special structures of varieties (see, e.g., \cite{Deligne-70-zbMATH03385791, gunning-67-zbMATH03233043, kobayashi-nagano-64-zbMATH03191725}).

From the preceding construction, two natural problems arise:
\begin{enumerate}[label = (\roman*)]
    \item\label{I: problem, algebraic} to provide an \emph{algebraic} construction of the jets of flat partial connections, and
    \item\label{I: problem, singular} to extend the definition for the context of \emph{singular} foliations and \emph{not necessarily locally free} sheaves.
\end{enumerate}
Indeed, in the construction of the sheaf of transverse jets we described relies deeply on both the ring of first integrals and the sheaf of flat sections of a flat connection, and these objects are not well-behaved in the settings presented above for several reasons. For instance, in general, a flat partial connection does not admit flat algebraic sections, and a foliation may not admit any algebraic first integral. Even in the analytic setting, one could not, in general, expect to find first integrals and flat sections on neighborhoods of singularities.

In this work, we provide a construction of jets of flat partial connections that solves Problem \ref{I: problem, algebraic} and, in the spirit of Problem \ref{I: problem, singular}, extend the construction for the setting of singular foliations and flat partial connections on \emph{reflexive} sheaves. This construction is presented in Section \ref{Subsection: transverse jets} where, in particular, the discussion above appears as Corollary \ref{Cor: transverse jets}.

\subsection{Restriction and extensions of partial connections} For every effective divisor $D$ on a smooth variety $X$, and every coherent sheaf $\mcE$, we denote by 
\[
    \Omega_X^1(D)(\mcE) := \Omega_X^1(D) \otimes \mcE
\]
the sheaf of meromorphic 1-forms with poles $D$ and with coefficients in $\mcE$. We say that a meromorphic connection 
\[
    \nabla: \mcE \rightarrow \OmegaDE
\]
\emph{restricts to a partial connection} if, for every $v\in T_{\F}$ and $s\in \mcE$, the element $\nabla_v(s) \in \mcE(D)$ is in fact a holomorphic section of the sheaf $\mcE$. That is, $\nabla$ induces a partial connection $\nabla_0$ on $\mcE$ such that the diagram
\begin{equation}\label{D: restriction of connections}
    \begin{tikzpicture}[baseline=(current bounding box.center)]
        \matrix(m)[matrix of math nodes, column sep = 2em, row sep = 2em]
        {
        \mcE & \OmegaDE \\
        \CTF^1(\mcE) & \CTF^1(\mcE(D)) \\
        };
        \path[->]
        (m-1-1) edge node[above]{$\nabla$} (m-1-2) edge node[left]{$\nabla_0$} (m-2-1)
        (m-1-2) edge node[left]{$\mrrestr$}(m-2-2)
        (m-2-1) edge node[above]{$\iota$} (m-2-2)
        ;
    \end{tikzpicture}
\end{equation}
commutes, where $\iota:\CTF^1(\mcE) \rightarrow \CTF^1(\mcE(D))$ is the natural inclusion induced by $D$, and $\mrrestr:\OmegaDE \rightarrow \CTF^1(\mcE(D))$ is the restriction of 1-forms induced by the inclusion $T_{\F} \rightarrow T_X$. Moreover, observe that if $\nabla$ restricts to a partial connection, then $\nabla_0$ in Diagram (\ref{D: restriction of connections}) is unique, and we say that $\nabla_0$ is the \emph{restriction} of $\nabla$, while $\nabla$ is an \emph{extension} of $\nabla_0$. 

Finding (flat) meromorphic extensions of partial connections is deeply connected to several aspects of the theory of holomorphic foliations, such as the description of transversely affine and transversely projective structures. This phenomenon appears in works such as \cite{cerveau-lins-neto-loray-pereira-touzet-07-zbMATH05202834, cousin-pereira-14-zbMATH06399599,fazoli25-arXiv:2505.11662,loray-pereira-touzet-16-zbMATH06670708}. 

In this work, we use the construction of jets of flat partial connections 
to characterize meromorphic extensions of a given flat partial connection $(\mcE,\nabla)$ in terms of splittings of a natural short exact sequence associated with $\PXF^1(\nabla)$ (see Equation \ref{Eq: short exact sequence of the first sheaf of transverse jets}). Precisely, we prove the following, which is the main result of this work (see Theorems \ref{T: meromorphic extensions} and \ref{T: flat meromorphic extensions}).

\begin{THM}\label{THM: meromorphic extension}
    Let $\F$ be a foliation on a smooth variety $X$, and let $(\mcE,\nabla)$ be a flat partial connection on a reflexive sheaf. Let $D\ge 0$ be a divisor on $X$. Then there is a natural correspondence (described in Proposition \ref{P: extensions and splittings}) between:
\begin{enumerate}[label = (\alph*)]
    \item meromorphic connections with poles on $D$ extending $\nabla$; and
    \item meromorphic splittings of the short exact sequence associated with the first transverse jet of $\nabla$ (see Equation (\ref{Eq: short exact sequence of the first sheaf of transverse jets})).
\end{enumerate}
If, in addition, $\F$ is of codimension one and $D$ is $\F$-invariant, then the above correspondence restricts to a correspondence between:
\begin{enumerate}[label = (\alph*')]
    \item flat meromorphic connections extending $\nabla$ with poles on $D$; and
    \item horizontal meromorphic splittings with poles on $D$ of the short exact sequence associated to the first transverse jet of $\nabla$.
\end{enumerate}
\end{THM}

The terminology will be explained in Section \ref{Section: extensions}. Observe that the first part of Theorem \ref{THM: meromorphic extension} generalizes the classical result of Atiyah, which characterizes the existence of holomorphic connections on vector bundles in terms of their sheaf of first jets (see \cite[Theorem 5]{atiyah-57-zbMATH03128044}), and the second part generalizes \cite[Theorem 5.4]{fazoli25-arXiv:2505.11662}, concerning extensions of flat partial connections in the setting of smooth foliations and locally free sheaves.

\subsection{The Bott connection and codimension one transversely affine foliations}

Recall that the \emph{conormal sheaf} of a foliation $\F$ is defined as
\[
\CNF := \{ \omega \in \Omega_X^1 \mid \omega(v)=0 \text{ for all } v\in T_{\F} \}, 
\]
which, by the short exact sequence induced by $T_{\F}$, is naturally isomorphic to $(T_X/T_{\F})^*$. We define $\rm{rk}(\CNF)$ to be the \emph{codimension of $\F$}. Finally, we define the \emph{normal sheaf} of $\F$ as $N_{\F} := (\CNF)^*$.

Each of $T_X/T_{\F}$, $\CNF$, and $N_{\F}$ carries a canonical flat partial connection, called the \emph{Bott connection} (see \cite[Example 3.5]{fazoli25-arXiv:2505.11662}). Explicitly, on $\CNF$, the Bott connection is given by
\begin{equation*}
    \begin{split}
        \nabla_B: \CNF & \rightarrow \CTF^1(\CNF) \\
        \omega & \mapsto \big(v\in T_{\F} \mapsto \mc{L}_v(\omega) \big).
    \end{split}
\end{equation*}

Suppose $\F$ has codimension one. Let $\Aff(\C)$ be the group of affine transformations of the complex line $\C$, and let $\aff(\C)$ denote its Lie algebra. Let $\C\subset \aff(\C)$ the subalgebra corresponding to the subgroup $\C^* \subset \Aff(\C)$ of affine transformations fixing $0\in \C$. Recall that a \emph{singular transversely affine structure} for $\F$ is a collection of $\aff(\C)$-valued meromorphic 1-forms
\[
    \mc C = \{ \Omega_i: T_{U_i} \rightarrow \germe_{U_i}(D) \otimes \aff(\C) \},
\]
satisfying the following conditions:
\begin{enumerate}[label = (\roman*)]
    \item $\mc U = \{ U_i\}$ is a covering of $X$;
    \item\label{Item: aff(C)-valued 1-forms that are flat and induces the foliation} for each $i$, $\Omega_i$ is flat and the kernel of the induced morphism
    \[
        T_{U_i} \rightarrow \aff(\C)/\C \otimes \germe_{U_i}(D)
    \]
    is $\restr{T_{\F}}{U_i}$; and
    \item\label{Item: compatibility of local affine structures} for every pair $(i,j)$ with $U_i\cap U_j \neq \emptyset$, there exists a holomorphic map $g_{ij}:U_i\cap U_j \rightarrow \C^*$ such that 
    \begin{equation*}
        \Omega_i = \mathrm{Ad}(g_{ij}^{-1}) \circ \Omega_j + g_{ij}^*(\Omega_{\C^*}),
    \end{equation*}
    where $\rm{Ad}: \Aff(\C) \rightarrow \Aut(\aff(\C))$ is the adjoint representation and $\Omega_{\C^*}$ is the Maurer-Cartan form of the Lie group $\C^*$ (see \cite[Chapter 3, Definition 1.3]{sharpe-97-zbMATH00914851}).
\end{enumerate}
Recall that this is just a particular case of a transversely homogeneous structure, as explained in \cite[Section 2.4]{fazoli25-arXiv:2505.11662}.

Let us show how a singular transversely affine structure can be described as a flat meromorphic extension of the Bott connection. Consider the basis $\{e_1,e_2\}$ for the Lie algebra $\aff(\C)$ such that $[e_1,e_2]= -e_2$, with $e_1$ generating $\C \subset \aff(\C)$. Each meromorphic $\aff(\C)$-valued 1-form $\Omega_i=(\eta_i,\omega_i)$ satisfies \ref{Item: aff(C)-valued 1-forms that are flat and induces the foliation} if and only if,
\begin{equation}\label{Eq: transversely affine structure}
    \omega_i \in \CNF, d\omega_i = \eta_i \wedge \omega_i, \text{ and } d\eta_i =0;
\end{equation}
while the collection $\{\Omega_i\}$ satisfies the compatibility condition \ref{Item: compatibility of local affine structures} if and only if,
\begin{equation}\label{Eq: compatibility of local affine structures, 1-forms}
    \omega_i = \frac{1}{g_{ij}}\cdot \omega_j \text{ and } \eta_i = \eta_j - \frac{dg_{ij}}{g_{ij}}.
\end{equation}
Thus, Equation (\ref{Eq: compatibility of local affine structures, 1-forms}) encodes a meromorphic connection $\hatnabla$ on the conormal sheaf $\CNF$, given by $\hatnabla(\omega_i) = \eta_i \otimes \omega_i$, while Equation (\ref{Eq: transversely affine structure}) means that this meromorphic connection is flat and extends the Bott connection. 

For that reason, throughout this work, we adopt the point of view of \cite{cousin-pereira-14-zbMATH06399599} and consider a singular transversely affine structure for $\F$ as a flat meromorphic extension of the Bott connection. With this in mind, the following result is a direct consequence of Theorem \ref{THM: meromorphic extension}.

\begin{COR}\label{COR: transversely affine structures}
    Let $\F$ be a codimension one foliation on a smooth complex variety $X$. Let $D\ge 0$ be a divisor on $X$. Then, the correspondence of Theorem \ref{THM: meromorphic extension} induces a natural correspondence between:
    \begin{enumerate}[label = (\alph*)]
        \item transversely affine structures for $\F$ with poles on $D$; and
        \item horizontal meromorphic splittings with poles on $D$ of the short exact sequence associated with the first sheaf of transverse jets of the Bott connection.
    \end{enumerate}
\end{COR}
Remark that the above corollary generalizes \cite[Corollary 5.5]{fazoli25-arXiv:2505.11662}, which concerns \emph{smooth} transversely affine structures for smooth foliations.

\subsection{Mendson-Pereira Conjecture}

In \cite[Conjecture 12.1]{mendson-pereira-22-arXiv:2207.08957}, W. Mendson and J.V. Pereira proposed the following conjecture about codimension one foliations on $\Pj^n$.

\begin{CONJ} Let $\F$ be a codimension one foliation on $\Pj^n$, $n\ge 3$. If $h^0(\Pj^n,\CTF^1)=0$, then $\F$ is a rational foliation.
\end{CONJ}

In this same article, the authors established partial results towards this conjecture (see \cite[Propositions 12.2 and 12.3]{mendson-pereira-22-arXiv:2207.08957}). Specifically, under the additional assumption of the existence of one (respectively, two) $\F$-invariant hypersurfaces, they proved that $\F$ is virtually transversely additive (respectively, rational foliation). In both cases, their argument relies on constructing, using these $\F$-invariant hypersurfaces, suitable flat meromorphic lifts of the Bott connection.

In this work, we provide a result relating the arguments in \cite{mendson-pereira-22-arXiv:2207.08957} with algebraic properties of the first sheaf of transverse jets of the Bott connection on the conormal sheaf.

\begin{THM}\label{THM: mendson-pereira conjecture}
    Let $\F$ be a codimension one foliation on $\Pj^n$, $n\ge 3$. Suppose that $h^0(\Pj^n, \CTF^1)=0$. Let $(\CNF, \nabla_B)$ be the Bott connection on the conormal sheaf. 
    \begin{enumerate}[label = (\Alph*)]
        \item If $\mc{P}_{\Pj^n/\F}^1(\nabla_B)$ is unstable, then $\F$ is virtually transversely additive.
        \item If $(\mc{P}_{\Pj^n/\F}^1(\nabla_B), \nabla_B^1)$ is decomposable (that is, isomorphic to the direct sum of two flat partial connections on line bundles), then $\F$ is a rational foliation.
    \end{enumerate}
\end{THM}

\subsection{Pencil of curves}

Let $C_1=\{F_1=0\}, C_2=\{F_2=0\} \subset \Pj^2$ be two curves with no common irreducible component, and of the same degree $d$, in the projective plane. Recall that the \emph{pencil generated by $C_1$ and $C_2$} is the family of curves
\[
C_{(a:b)} := \{a\cdot F_1+b\cdot F_2 =0 \mid (a:b)\in \Pj^1 \} \subset \Pj^2.
\]
This pencil naturally induces a foliation of degree $2d-2$ on $\Pj^2$, such that its leaves are contained in the member of the pencil. We refer to this foliation as the \emph{pencil generated by $C_1$ and $C_2$} as well.

For a generic choice of curves, the singularities of the pencil generated by $C_1$ and $C_2$ are of two types only: there are $d^2$ radial singularities determined by the intersection of $C_1$ and $C_2$, while the remaining singularities are locally given by a closed 1-form. 

Motivated by this behavior, we say that a foliation $\F$ on $\Pj^2$ \emph{has singularities of pencil type} if every singularity $p\in \sing(\F)$ is of one of the following types:
\begin{enumerate}[label = -]
    \item in a neighborhood of $p$, there  exists a system of coordinates $(x,y)$ where $\F$ is given by the 1-form $\omega = xdy - ydx$ (that is, a radial singularity); or
    \item in a neighborhood of $p$, $\F$ is given by a closed holomorphic 1-form.
\end{enumerate}

In this work, we prove the following results about foliations with singularities of pencil type.

\begin{THM}\label{THM: pencil 1}
    Let $\F$ be a degree $d$ foliation on $\Pj^2$ with singularities of pencil type. If $\F$ admits at most $(d+2)^2/4$ radial singularities, then $\F$ admits an $\F$-invariant curve of degree at most $(d+2)/2$ passing through all the radial singularities.
\end{THM}

\begin{THM}\label{THM: pencil 2}
    Let $\F$ be a foliation of degree $2d-2$ on $\Pj^2$. Suppose that $\F$ is a foliation with singularities of pencil type, with exactly $d^2$ radial singularities determined by the intersection of two curves $C_1$ and $C_2$ of degree $d$. Additionally, suppose that no curve of degree less than $d$ passes through all the radial singularities. Then $\F$ is the pencil of curves determined by $C_1$ and $C_2$.
\end{THM}

\subsection{Structure of the paper} This work is organized as follows. In Section \ref{Section: transverse differential operators} we briefly review the theory of differential operators between sheaves and define the sheaf of \emph{transverse differential operators between flat partial connections}.  In Section \ref{Section: jets of flat partial connections}, we recall the basic aspects of the theory of jets of sheaves and relate it to the theory of differential operators; we then define the sheaves of transverse jets of a flat partial connection on a reflexive sheaf and prove that they are naturally endowed with a flat partial connection. In Section \ref{Section: extensions}, we take a closer look at the first sheaf of transverse jets, relating it to the existence of (flat) meromorphic extensions of a flat partial connection. In Section \ref{S: foliations with no global foliated 1-forms}, we apply the theory we developed to study foliations on projective spaces without global foliated 1-forms, obtaining results in the direction of the Mendson-Pereira Conjecture. Finally, in Section \ref{S: Codimension one foliation on the projective plane with certain singularities}, we apply the presented theory to the problem of characterizing pencil of curves on the projective plane; in this context, we show that a generic foliation with singularities of pencil type are indeed pencil of curves.

\subsection{Acknowledgments} This work would not exist without the guidance of my PhD advisor, Jorge Vitório Pereira, to whom I am deeply grateful. I also thank João Pedro dos Santos, who patiently discussed preliminary versions of this work with me and helped with several suggestions of all kinds. I acknowledge the support of FAPERJ (Grant number E26/201.353/2023) and CAPES (Grant number 88887.184306/2025-00).

\section{Transverse differential operators} \label{Section: transverse differential operators}

\subsection{Differential operators} \label{Subsection: differential operators}
Let $X$ be a variety, and let $\mcE,\mcE'$ be sheaves of $\germe_X$-modules. Let us define differential operators from $\mcE$ to $\mcE'$ inductively as follows. A $\C$-linear morphism $D: \mcE\rightarrow \mcE'$ is a \emph{differential operator of order $\le k$} if:
\begin{enumerate}[label = -]
    \item for $k=0$, $D$ is $\germe_X$-linear;
    \item for $k\ge 1$, for every $f\in \germe_X$, the $\C$-linear morphism
    \begin{equation*}
        \begin{split}
            [D,f] : \mcE & \rightarrow \mcE' \\
            s & \mapsto D(f\cdot s) - f \cdot D(s) \\
        \end{split}
    \end{equation*}
    is a differential operator of order $\le k-1$.
\end{enumerate}
We denote by $\DX{k}(\mcE,\mcE')$ the sheaf of differential operators of order $\le k$ from $\mcE$ to $\mcE'$. This sheaf is naturally endowed with two structures of $\germe_X$-module. The first one, called the \emph{left} or the \emph{canonical} structure, is given by
\[
(f \cdot D)(s) := f \cdot D(s), \forall f\in \germe_X, D\in \DX{k}(\mcE,\mcE'), s\in \mcE.
\]
The second structure, referred to as the \emph{right} structure, is defined by
\[
(D\cdot f)(s) := D (f \cdot s) , \forall f\in \germe_X, D\in \DX{k}(\mcE,\mcE'), s\in \mcE.
\]
In the case where $\mcE$ are $\mcE'$ coherent sheaves, we have that $\DX{k}(\mcE,\mcE')$ is also coherent with respect to both canonical and right structures (see \cite[Proposition 16.8.6]{ega-VI-zbMATH03245973}).

\begin{rmk}
    In Section \ref{Subsection: jets and differential operators} we provide an alternative definition of the sheaf of differential operators, which is the standard definition appearing both in \cite[Definition 2.1]{berthelot-78-zbMATH03595321} and \cite[Definition 16.8.1]{ega-VI-zbMATH03245973}. The equivalence of these definitions can be found in \cite[Proposition 16.8.8]{ega-VI-zbMATH03245973}. We adopt the recursive definition here, as it is the most suitable for proving the main theorem of this section, namely Theorem \ref{T: connection on the transverse differential operators}. 
\end{rmk}

\subsection{Composition of differential operators}

Let $\mcE,\mcE',\mcE''$ be sheaves of $\germe_X$-modules. Given differential operators $D\in \DX{k}(\mcE,\mcE')$ and $D'\in \DX{l}(\mcE',\mcE'')$, their composition defines a $\C$-linear morphism $D'\circ D: \mcE\rightarrow \mcE'$. Observing that
\[
    [D'\circ D,f] = [D',f] \circ D + D \circ [D',f], \forall f\in \germe_X,
\]
it follows by induction that $D'\circ D $ is a differential operator of order $\le k+l$. Moreover, for every $f\in \germe_X$, we have that $(D'\cdot f) \circ D = D' \circ(f \cdot D)$. Therefore, the composition of differential operators corresponds to the $\germe_X$-linear morphism
\begin{equation*}
    \begin{split}
        \delta_{l,k}: \DX{l}(\mcE',\mcE'') \otimes \DX{k}(\mcE, \mcE') & \rightarrow \DX{k+l}(\mcE,\mcE'') \\
        D' \otimes D & \mapsto D' \circ D,
    \end{split}
\end{equation*}
where the tensor product is taken with respect to the canonical $\germe_X$-module structure of $\DX{k}(\mcE, \mcE')$, and the right $\germe_X$-module structure of $\DX{l}(\mcE',\mcE'')$.

\subsection{Partial connections as differential operators}
Let $\F$ be a foliation on a complex smooth variety $X$, and let $(\mcE,\nabla)$ be a partial connection on a coherent sheaf. Since $\nabla$ satisfies the Leibniz rule (see Equation (\ref{Eq: Leibniz rule})), it follows that, for every $v\in T_{\F}$, the morphism $\nabla_v:\mcE \rightarrow \mcE$  is a differential operator of order $\le 1$. Thus, $\nabla$ induces a $\germe_X$-linear morphism (with respect to the canonical structure)
\begin{equation*}
    \begin{split}
        \nabla: T_{\F} & \rightarrow \DX{1}(\mcE,\mcE) \\
        v & \mapsto \nabla_v,
    \end{split}
\end{equation*}
which is injective since $\nabla_v \neq 0$ for every non-zero $v\in T_{\F}$. This allow us to see $T_{\F}$ as a $\germe_X$-submodule of $\DX{1}(\mcE,\mcE)$, but only with respect to the canonical structure of $\DX{1}(\mcE,\mcE)$, as the image of $T_{\F}$ under $\nabla$ is not preserved by the right multiplication on $\DX{1}(\mcE,\mcE)$. 

The flatness of $\nabla$ can also be described in terms of this inclusion. Note that $\DX{1}(\mcE,\mcE)$ is endowed with a natural Lie algebra structure, given by
\[
[D,D'] := D \circ D' - D' \circ D.
\]
Using this structure, it follows directly from the definition (see Equation (\ref{Eq: flatness})) that $\nabla$ is flat if and only if the inclusion $T_{\F} \rightarrow \DX{1}(\mcE,\mcE)$ is a morphism of Lie algebras.

\subsection{The sheaf of transverse differential operators}

Let $\F$ be a foliation on a complex smooth variety $X$, and let $(\mcE,\nabla)$ be a partial connection on a coherent sheaf. For every coherent sheaf $\mcE'$ and $k\ge 1$, we define the \emph{$k$-th prolongation of $\nabla$ with respect to $\mcE'$} as the system of differential operators
\begin{equation*}
    \begin{split}
        S^k(\nabla) & := \left\{ \sum D \circ \nabla_v \mid D\in \DX{k-1}(\mcE,\mcE'), v\in T_{\F}\right\} \subset \DX{k}(\mcE,\mcE') \\
        & = \mathrm{im}\Big(\delta_{k-1,1} \circ (\mathrm{id} \otimes \nabla): \DX{k-1}(\mcE,\mcE') \otimes T_{\F} \rightarrow \DX{k}(\mcE,\mcE')\Big)
    \end{split}
\end{equation*}
Additionally, we set $S^0(\nabla)=0$. By definition, $S^k(\nabla)$ is an $\germe_X$-submodule of $\DX{k}(\mcE,\mcE')$ (with respect to the canonical structure),  and since both $\mcE,\mcE'$ are coherent, it follows that $S^k(\nabla)$ is coherent as well.

For every $k\ge 0$, we define the \emph{sheaf of transverse differential operators from $(\mcE,\nabla)$ to $\mcE'$ of order $\le k$} as the quotient
\[
\DXF{k}(\nabla, \mcE') := \frac{\DX{k}(\mcE,\mcE')}{S^k(\nabla)}.
\]
For every $D\in \DX{k}(\mcE,\mcE')$, we abuse notation and denote its image by the natural projection $\DX{k}(\mcE,\mcE') \rightarrow \DXF{k}(\nabla,\mcE')$ also by $D$.  

\begin{rmk}
An element on $\DXF{k}(\nabla, \mcE')$ is not, strictly speaking, a well-defined differential operator on $\mcE$, since it arises from a quotient of $\DX{k}(\mcE,\mcE')$. Nevertheless, in the analytic category, restricting to the locus where $\F$ is smooth and $\mcE$ is locally free, an element $D\in \DXF{k}(\nabla, \mcE')$ defines, in the category of $\germe_{X/\F}$-modules, a  differential operator of order $\le k$
\[
D: \ker \nabla' \rightarrow \mcE'.
\]
Although our primary interest is in the sheaf $\DXF{k}(\nabla,\mcE')$ over the singular locus of the foliation (as explained in the introduction), this justifies the terminology of transverse differential operators for elements of $\DXF{k}(\nabla, \mcE')$.
\end{rmk}

\begin{ex}\label{Ex: first sheaf of transverse differential operators}
    Let $(\mcE,\nabla)=(\germe_X,\df + \eta \otimes \rm{id})$ for some $\eta \in H^0(\CTF^1)$, and let $\mcE' = \germe_X$. The natural short exact sequence
    \begin{equation}\label{Eq: short exact sequence of the sheaf of differential operators of order 1}
        0 \rightarrow \DX{0}(\germe_X,\germe_X) \simeq \germe_X \rightarrow \DX{1}(\germe_X,\germe_X) \rightarrow T_X \rightarrow 0
    \end{equation}
    induces the short exact sequence
    \begin{equation}\label{Eq: short exact sequence of the first sheaf of transverse differential operators}
        0\rightarrow \germe_X \rightarrow \DXF{1}(\df +\eta,\germe_X) \rightarrow T_X/T_{\F} \rightarrow 0.
    \end{equation}
    We claim that this short exact sequence is splitting if, and only if, $\eta$ belongs to the image of the restriction morphism $\mathrm{restr}: H^0(\Omega_X^1) \rightarrow H^0(\CTF^1)$. Indeed, if there exist $\eta_0\in H^0(\Omega_X^1)$ such that $\eta = \rm{restr}(\eta_0)$, then the morphism 
    \[
    \begin{split}
        \iota + \eta_0: T_X & \rightarrow \DX{1}(\germe_X,\germe_X) \\
        v & \mapsto v + \eta_0(v)
    \end{split}
    \]
    induces a splitting $T_X/T_{\F} \rightarrow \DXF{1}(\df +\eta,\germe_X)$ of the short exact sequence of Equation (\ref{Eq: short exact sequence of the first sheaf of transverse differential operators}). Conversely, a splitting $\DXF{1}(\germe_X,\germe_X) \rightarrow \germe_X$ of Equation (\ref{Eq: short exact sequence of the first sheaf of transverse differential operators})  induces a splitting $\DX{1}(\germe_X,\germe_X) \rightarrow \germe_X$ of Equation (\ref{Eq: short exact sequence of the sheaf of differential operators of order 1}), which must be of the form
    \[
    \begin{split}
        \phi: \DX{1}(\germe_X,\germe_X)\simeq T_X \oplus \germe_X & \rightarrow \germe_X \\
        v + f & \mapsto \eta_0(v) + f
    \end{split}
    \]
    for some $\eta_0 \in H^0(\Omega_X^1)$. Since $\phi(S^1(\df + \eta \otimes \rm{id}))=0$, it follows that $\mathrm{restr}(\eta_0) = \eta$.
\end{ex}

\subsection{The natural partial connection on the sheaf of transverse differential operators}

\begin{thm}\label{T: connection on the transverse differential operators}
    Let $\F$ be a foliation on a complex smooth variety $X$. Let $(\mcE,\nabla)$ be a flat partial connection on a coherent sheaf, and let $\mcE'$ be a coherent sheaf. Then, for every partial connection $\nabla'$ on $\mcE'$, there exists a partial connection $(\nabla,\nabla')^k$ on $\DXF{k}(\nabla,\mcE')$ defined by
    \[
    (\nabla,\nabla')^k_v(D) = \nabla_v' \circ D - D \circ \nabla_v , 
    \]
    for every $D \in \DXF{k}(\nabla,\mcE')$ and $v\in T_{\F}$. Moreover, if $\nabla'$ is flat, then $(\nabla,\nabla')^k$ is also flat.
\end{thm}
The proof of Theorem \ref{T: connection on the transverse differential operators} is given in Section \ref{S: proof of theorem}.

\begin{rmk}\label{Rmk: expression for the fpc on transverse diff ope}
    If $D = \phi_1 \circ w_1\circ \cdots \circ w_l \circ \phi_2 \in \DX{k}(\mcE,\mcE')$ for $\phi_1\in \mcHom(\germe_X, \mcE')$, $w_i\in T_X = \mathrm{Der}(\germe_X)$ and $\phi_2\in \mcHom(\mcE,\germe_X)$, a straightforward calculation shows that
    \begin{multline*}
        (\nabla,\nabla')^k_v(D) = \nabla'_v(\phi_1) \circ w_1\circ \cdots \circ w_l \circ \phi_2  + \sum_{i=1}^l \big(\phi_1 \circ w_1\circ \cdots \circ \Lie_v(w_i) \circ \cdots \circ w_l\big) \\ + \phi_1 \circ w_1\circ \cdots \circ w_l \circ \nabla_v^*(\phi_2).
    \end{multline*}
    This expression will be useful later in this text.
\end{rmk}

We denote the pair $(\DXF{k}(\nabla,\mcE'),(\nabla,\nabla')^k )$ by $\DXF{k}((\mcE,\nabla), (\mcE',\nabla'))$, and call it the \emph{sheaf of differential operators from $(\mcE,\nabla)$ to $(\mcE',\nabla')$ of order $\le k$}. Let us illustrate this construction through some examples.

\begin{ex} For any pair of flat partial connections $(\mcE,\nabla), (\mcE',\nabla')$, the sheaf of transverse differential operators of order $\le 0$ is simply $\mcHom(\mcE,\mcE')$, and the partial connection $(\nabla,\nabla')^{0}$ is given by 
\[
(\nabla,\nabla')^{0}_v(\phi) = \nabla'_v \circ \phi - \phi \circ \nabla_v, \forall \phi \in \mcHom(\mcE,\mcE'), v\in T_{\F}.
\]
That is, $(\nabla,\nabla')^{0}$ is just the natural partial connection on $\mcHom(\mcE,\mcE')$ (see \cite[Section 3.4]{fazoli25-arXiv:2505.11662}).
\end{ex}

\begin{ex} Let $(\mcE,\nabla) = (\mcE',\nabla') = (\germe_X,\df)$. The natural isomorphism $\DX{1}(\germe_X,\germe_X) \simeq T_X \oplus \germe_X$ induces an isomorphism 
\[
    \DXF{1}(\df, \germe_X) \simeq (T_X/T_{\F}) \oplus \germe_X.
\]
A direct calculation shows that, under this isomorphism, $(\df,\df)^1 = \nabla_B \oplus \df$, where $\nabla_B$ stands for the Bott connection on $T_X/T_{\F}$ (see \cite[Example 3.5]{fazoli25-arXiv:2505.11662}).
\end{ex}

\begin{ex}\label{Ex: connection on first sheaf of transverse differential operators}
    Following Example \ref{Ex: first sheaf of transverse differential operators}, let us consider the connection $\nabla' = \df$ on $\germe_X$. For every $v\in T_{\F}$ and $w+f \in \DX{1}(\germe_X,\germe_X)$, we compute that
    \begin{equation*}
        \begin{split}
            v \circ (w+f) - (w+f) \circ (v + \eta(v)) & = [v,w] + v(f) - (w+ f)\circ \eta(v)  \\
            & = [v,w] + v(f) - w(\eta(v)) - \eta(v) \cdot w - f \cdot \eta(v)
        \end{split}
    \end{equation*}
    Thus, $\DXF{1}(\df+ \eta\otimes \rm{id}, \germe_X)$ fits into a short exact sequence of partial connections:
    \[
    0 \rightarrow (\germe_X,\df - \eta \otimes \rm{id}) \rightarrow \DXF{1}(\df+\eta\otimes \rm{id}, \df) \rightarrow (T_X/T_{\F}, \nabla_B- \eta \otimes \rm{id}) \rightarrow 0,
    \]
    where in the category of sheaves this is the same short exact sequence of Equation (\ref{Eq: short exact sequence of the first sheaf of transverse differential operators}).
    Recall that the short exact sequence above is splitting (in the category of sheaves) if, and only if, $\eta \in \rm{im}\left(\mathrm{restr: }   H^0(\Omega_X^1) \rightarrow H^0(\CTF^1)\right)$. However, generally it is not splitting in the category of flat partial connections. For instance, for codimension one foliations, a straightforward calculation shows that the splitting determined by $\eta_0 \in H^0(\Omega_X^1)$ (as in the notation of Example \ref{Ex: first sheaf of transverse differential operators}) is horizontal if and only if $\eta_0$ is closed.
\end{ex}

\begin{rmk}
    The example above anticipates the results presented in Section \ref{Section: extensions}, where we relate splittings of a natural short exact sequence associated with the first sheaf of transverse jets of a flat partial connection to the existence of (flat) extensions (see Theorems \ref{T: meromorphic extensions} and \ref{T: flat meromorphic extensions}).
\end{rmk}

The Examples \ref{Ex: first sheaf of transverse differential operators}  and \ref{Ex: connection on first sheaf of transverse differential operators} illustrate how both the sheaf $\DXF{k}(\nabla,\mcE')$ and the connection $(\nabla,\nabla')^k$ depend on the connection $\nabla$ on $\mcE$. On the other hand, since the connection $\nabla'$ on $\mcE'$ is not involved in the construction of the sheaves $\DXF{k}(\nabla,\mcE)$, it remains to understand how it affects the partial connection $(\nabla,\nabla')^k$.

Given two partial connections $\nabla_1',\nabla_2'$ on $\mcE'$, their difference $\phi = \nabla_1' - \nabla_2'$ is an $\germe_X$-linear morphism from $\mcE'$ to $\CTF^1(\mcE')$. The composition of $\phi$ with a differential operator $D\in \DX{k}(\mcE,\mcE')$ induces an element $D \circ \phi \in \DX{k}(\mcE, \CTF^1(\mcE'))$, which is canonically identified with an element of $\CTF^1(\DX{k}(\mcE,\mcE')$. Hence, $\phi$ induces an $\germe_X$-linear morphism
\begin{equation*}
    \begin{split}
        \DX{k}(\mcE,\mcE') & \rightarrow \CTF^1(\DX{k}(\mcE,\mcE'))
    \end{split}
\end{equation*}
which, in turn, induces an $\germe_X$-linear morphism 
\begin{equation*}
    \begin{split}
        \tilde{\phi}: \DXF{k}(\nabla,\mcE') &  \rightarrow \CTF^1(\DXF{k}(\nabla,\mcE')) \\
        D & \mapsto \Big(v\in T_{\F} \mapsto (\nabla_1'-\nabla_2')_v \circ D \Big)
    \end{split}
\end{equation*}
Directly from the definition, it follows that
\[
    (\nabla,\nabla_1')^k = (\nabla,\nabla_2')^k + \tilde{\phi}.
\]

\begin{ex}
    Let $\eta \in H^0(\CTF^1)$ be a global foliated 1-form. Then, for every pair of partial connections $(\mcE,\nabla)$ and $(\mcE',\nabla')$, we have that
    \[
    (\nabla,\nabla' + \eta \otimes \mathrm{id})^k = (\nabla,\nabla')^k \otimes (\germe_X, \df + \eta \otimes \rm{id}). 
    \]
\end{ex}

\subsection{Proof of Theorem \ref{T: connection on the transverse differential operators}}\label{S: proof of theorem}

The first step is to verify that 
\begin{equation*}
    \begin{split}
        (\nabla,\nabla')^k_v: \DXF{k}(\nabla, \mcE') & \rightarrow \DXF{k}(\nabla, \mcE') \\
    D & \mapsto \nabla_v' \circ D - D\circ \nabla_v
    \end{split}
\end{equation*}
is well defined. This will follows directly from the two claims below.

\begin{claim}\label{Claim: constructions of the connections 1} Let $D\in \DX{k}(\mcE, \mcE')$ be a differential operator of order $\le k$, and $v \in T_{\F}$ be a vector tangent to the foliation.  Then 
\[
    [v,D] := \nabla_v' \circ D - D\circ \nabla_v
\]
is also a differential operator of order $\le k$.
\end{claim}
\begin{proof}
   We proceed by induction on $k$. For $k=0$, observe that $D$ is $\germe_X$-linear. Thus, for every $f\in \germe_X$ and $s\in \mcE$, we have that 
\begin{align*}
[v,D](fs) & = \nabla_v' (f\cdot D(s)) - D(v(f) \cdot s + f \cdot \nabla_v(s)) \\
& =  v(f) \cdot D(s) + f \cdot \nabla'_v \circ D(s) - v(f)\cdot D(s) - f \cdot D\circ \nabla_v(s) \\
& =  f \cdot (\nabla_v' \circ D - D \circ \nabla_v)(s) = f\cdot [v,D](s).
\end{align*}
Hence, $[v,D]$ is $\germe_X$-linear and the claim holds for $k=0$.

Now assume that the claim holds for $k-1$, and consider the case $k$. For every $f\in \germe_X$, 
\begin{align*}
    [[v,D],f] & = [\nabla_v' \circ D,f] - [ D \circ \nabla_v,f] \\
    & = [\nabla_v',f]\circ D + \nabla_v' \circ [D,f] - [D,f] \circ \nabla_v -  D\circ [\nabla_v,f] \\
    & = -[D,v(f)] + [v,[D,f]].
\end{align*}
Since $[D,v(f)]$ is of order $\le k-1$ by definition, and $[v,[D,f]]$ is of order $\le k-1$ by the induction hypothesis, we conclude that $[[v,D],f]$ is of order $\le k-1$ for all $f\in \germe_X$. Then, $[v,D] \in \DX{k}(\mcE,\mcE')$ and the claim is true for $k$. Therefore, by induction, the claim holds for every $k\ge 0$.
\end{proof}

\begin{claim}\label{Claim: constructions of the connections 2} $[T_{\F}, S^k(\nabla)] \subset S^k(\nabla)$.
\end{claim}
\begin{proof}
    It is enough to verify the claim for elements of the form $D\circ \nabla_v$, where $D\in \DX{k}(\mcE,\mcE')$ and $v\in T_{\F}$. For every $w\in T_{\F}$,
    \begin{align*}
        [w,D \circ \nabla_v] & = \nabla'_w \circ D \circ \nabla_v - D \circ \nabla_v \circ \nabla_w \\
        & = \nabla'_w \circ D \circ \nabla_v - D \circ \nabla_w \circ \nabla_v + D \circ \nabla_w \circ \nabla_v -D \circ \nabla_v \circ \nabla_w \\
        & = [w,D] \circ \nabla_v + D \circ \nabla_{[w,v]},
    \end{align*}
    where the last equality follows from the flatness of $\nabla$. Since $[w,D]$ has order $\le k-2$, it follows that $[w,D \circ \nabla_v] \in S^k(\nabla)$. This concludes the proof.
\end{proof}

Observe that
\begin{equation*}\label{E: constructions of the connections 1}
    [fv,D] = f \cdot (\nabla_v' \circ D - D\circ \nabla_v) - [D,f] \circ \nabla_v, 
\end{equation*}
for all $f\in \germe_X$, $v\in T_{\F}$ and $D\in \DX{k}(\mcE,\mcE')$. Since $[D,f] \circ \nabla_v \in S^k(\nabla)$, it follows that
\[
    (\nabla,\nabla')^k_{fv}(D) = f \cdot (\nabla,\nabla')^k_v(D).
\]
Hence, $(\nabla,\nabla')_v^k$ is $\germe_X$-linear with respect to $v\in T_{\F}$, and therefore
\[
(\nabla,\nabla')^k: \DXF{k}(\nabla,\mcE') \rightarrow \CTF^1(\DXF{k}(\nabla,\mcE')
\]
is a well-defined $\C$-linear morphism. It remains to verify that $(\nabla,\nabla')^k$ satisfies the Leibniz rule. A straightforward calculation shows that
\begin{equation*}
    [v, f\cdot D] = v(f) \cdot D + f \cdot [v,D],
\end{equation*}
and thus
\[
    (\nabla,\nabla')^k_{v}(f\cdot D) = v(f) \cdot D + f \cdot (\nabla,\nabla')^k_{v}(D),
\]
which corresponds to the Leibniz rule. Therefore, $(\nabla,\nabla')^k$ is a partial connection. 

Finally, let us verify that, whenever $\nabla'$ is flat, the partial connection $(\nabla,\nabla')^k$ is also flat for every $k\ge 0$. For every $v,w\in T_{\F}$ and $D\in \DX{k}(\mcE,\mcE')$, we calculate that
\begin{equation*}
    \begin{split}
        [[v,w],D] & = \nabla_{[v,w]}' \circ D - D \circ \nabla_{[v,w]} = [\nabla'_v, \nabla'_w] \circ D - D \circ [\nabla_v,\nabla_w] \\
        & = (\nabla'_v \circ \nabla'_w - \nabla'_w \circ \nabla'_v) \circ D - D \circ (\nabla_v \circ \nabla_w - \nabla_w \circ \nabla_v) \\
        & = [v,[w,D]] -[w,[v,D]].
    \end{split}
\end{equation*}
Hence, $(\nabla,\nabla')^k_{[v,w]}(D) = [(\nabla,\nabla')^k_v, (\nabla,\nabla')^k_w](D)$, and therefore $(\nabla,\nabla')^k$ is flat. 

This concludes the proof of Theorem \ref{T: connection on the transverse differential operators}. \qed

\section{Jets of flat partial connections}\label{Section: jets of flat partial connections}

\subsection{Jets} The following definitions can be found in \cite[Chapter 2]{berthelot-78-zbMATH03595321} and \cite[Section 16.7]{ega-VI-zbMATH03245973}. Although the theory of jets is established in a much more general setting, throughout this text we consider only jets of sheaves over a smooth complex variety $X$. The \emph{ring of $k$-jets} over $X$ is defined as the sheaf of rings
\[
\PX^k := \frac{\germe_X\otimes_{\C} \germe_X}{\mc I^{k+1}},
\]
where $\mc I$ denotes the kernel of the morphism of sheaves of rings $\germe_X\otimes_{\C} \germe_X \rightarrow \germe_X$ given by $f\otimes g \mapsto f \cdot g$. We abuse notation and denote by $f\otimes g \in \PX^k$ the image of $f\otimes g \in \germe_X \otimes_{\C} \germe_X$ under the natural projection $\germe_X \otimes_{\C} \germe_X \rightarrow \PX^k$.

The ring of $k$-jets is naturally endowed with two $\germe_X$-algebra structures. The first one, referred to as the \emph{canonical} structure, is induced by the ring morphism $\germe_X \rightarrow \PX^k$ given by $f\mapsto f\otimes 1$. If no confusion arises, we denote $f\otimes 1 \in \PX^k$ simply by $f$. The second $\germe_X$-algebra structure for $\PX^k$ is induced by the ring morphism $d^k_X: \germe_X \rightarrow \PX^k$ given by $f\mapsto 1\otimes f$, and is referred to as the \emph{right} structure. 

Additionally, let us keep the following convention: whenever we write $\otimes \PX^k$ (respectively, $\PX^k \otimes$ ), the tensor product is taken with respect to the canonical (respectively, right) $\germe_X$-algebra structure of $\PX^k$.

Let $\mcE$ be a $\germe_X$-module. For every $k\ge 0$, we define the \emph{sheaf of $k$-jets of sections of $\mcE$} as 
\[
\PX^k(\mcE) := \PX^k \otimes \mcE.
\]
Observe that, by definition, $\PX^k(\mcE)$ can be seen simultaneously as a:
\begin{enumerate}[label = -]
    \item $\PX^k$-module, regarding $\PX^k$ as a sheaf of rings; 
    \item $\germe_X$-module considering the canonical $\germe_X$-module structure of $\PX^k$; and
    \item $\germe_X$-module considering the right $\germe_X$-module structure of $\PX^k$.
\end{enumerate}
We refer to these $\germe_X$-module structures of $\PX^k(\mcE)$ as the \emph{canonical} and \emph{right}, respectively. For every $f\in \germe_X$, $\eta\in \PX^k$, and $s\in \mcE$, the \emph{canonical multiplication} is given by
    \[
    f \cdot (\eta \otimes s) = (f\cdot \eta) \otimes s.
    \]
while the \emph{right multiplication} is given by
    \[
    (\eta \otimes s) \cdot f = \eta \otimes (f \cdot s) = (d_X^k(f) \cdot \eta) \otimes s.
    \]
If $\mcE$ is coherent, then $\PX^k(\mcE)$ is coherent with respect to both $\germe_X$-module structures (see \cite[Proposition 16.7.3]{ega-VI-zbMATH03245973}).

Considering the right $\germe_X$-module structure of $\PX^k$, we have naturally define the $\germe_X$-linear morphism
\begin{equation*}
    \begin{split}
        d^k_{X,\mcE}: \mcE & \rightarrow \PX^k(\mcE) \\
        s & \mapsto 1\otimes s.
    \end{split}
\end{equation*}
When $X$ and $\mcE$ are implicit in the context, we denote $d^k_{X,\mcE}$ simply by $d^k$. For every section $s\in \mcE$, we refer to $d^k(s)$ as the \emph{$k$-jet of $s$}.

\subsection{The natural short exact sequence associated to the sheaf of jets}

Directly from the definition of the ring of jets, we deduce for every $k\ge 0$ the natural short exact sequence
\[
    0 \rightarrow \frac{I^{k+1}}{I^{k+2}} \rightarrow \PX^{k+1} \rightarrow \PX^k \rightarrow 0.
\]
Assume that $X$ is smooth. In this case, $\PX^k$ is locally free for every $k\ge 0$ with respect to both canonical and right $\germe_X$-module structures (see \cite[Proposition 2.2]{berthelot-78-zbMATH03595321}), and there exists a natural isomorphism
\begin{equation*}
    \begin{split}
        \Sym^{k+1}(\Omega_X^1) & \rightarrow \frac{I^{k+1}}{I^{k+2}} \\
        df_1 \cdots df_{k+1}  & \mapsto (1\otimes f_1 - f_1\otimes1)\cdots (1\otimes f_{k+1} - f_{k+1}\otimes 1)
    \end{split}
\end{equation*}
Hence, for every coherent sheaf $\mcE$, this induces the short exact sequence
\begin{equation}\label{Eq: short exact sequence of the sheaf of jets}
    0 \rightarrow \Sym^{k+1}(\Omega_X)(\mcE) \xrightarrow{\iota} \PX^{k+1}(\mcE) \xrightarrow{\pi} \PX^k(\mcE) \rightarrow 0,
\end{equation}
where $\iota$ is explicitly given by
\[
\iota(df_1\cdots df_{k+1}\otimes e) = (d^{k+1}_Xf_1-f_1)\cdots (d^{k+1}_Xf_{k+1} - f_{k+1})\otimes e.
\]
Equation (\ref{Eq: short exact sequence of the sheaf of jets}) is called the \emph{short exact sequence associated to $\PX^{k+1}(\mcE)$}. For every $\Omega \in \Sym^k(\Omega_X^1), s\in \mcE$, we abuse notation and denote $\iota(\Omega \otimes s)$ simply by $\Omega\otimes s$. This abuse is consistent with the notations introduced thus far.

\subsection{Jets and differential operators}\label{Subsection: jets and differential operators} 
The relation between jets and differential operators is established in the following proposition, which we state without proof (see \cite[Proposition 16.8.8]{ega-VI-zbMATH03245973}).
\begin{prop}
    Let $X$ be a variety, let $\mcE,\mcE'$ be $\germe_X$-modules, and let $k\ge 0$. Then, for every $\C$-linear morphism $D: \mcE \rightarrow \mcE'$, the following conditions are equivalent:
    \begin{enumerate}[label = (\alph*)]
        \item \label{Item: definition of differential operators of EGA} there exists an $\germe_X$-linear morphism $\phi_D: \PX^k(\mcE) \rightarrow \mcE'$ such that $D = \phi_D \circ d^k_{X,\mcE}$; and
        \item $D$ is a differential operator of order $\le k$ (as we defined in Section \ref{Subsection: differential operators}).
    \end{enumerate}
\end{prop}
Note that Item \ref{Item: definition of differential operators of EGA} corresponds to the definition of differential operators in \cite[Definition 16.8.1]{ega-VI-zbMATH03245973}. Throughout this work, we keep the notation of this item, denoting by $\phi_D$ the $\germe_X$-linear morphism such that $D = \phi_D \circ d^k_{X,\mcE}$.

As a consequence of the preceding proposition, we have the following natural isomorphism of sheaves of groups (see \cite[Proposition 16.8.4]{ega-VI-zbMATH03245973}):
\begin{equation}\label{Eq: jets and differential operators}
    \begin{split}
        \mcHom_{\germe_X}(\PX^k(\mcE),\mcE') & \rightarrow \DX{k}(\mcE,\mcE') \\
        \phi & \mapsto \phi \circ d^k_{X,\mcE}
    \end{split},
\end{equation}
where $\mcHom_{\germe_X}$ is with respect to the canonical $\germe_X$-module structure of $\PX^k(\mcE)$. Moreover, it follows directly from the definition that this isomorphism is $\germe_X$-linear with respect to both canonical and right structures of the respective sheaves.

\subsection{Jets of torsion-free and reflexive sheaves} 

\begin{prop}\label{P: jet of torsion-free sheaf}
    Let $\mcE$ be a torsion-free sheaf on a smooth complex variety $X$. Then, for all $k\ge 0$, the sheaf $\PX^k(\mcE)$ is also torsion-free .
\end{prop}
\begin{proof}
    We proceed by induction on $k$. When $k=0$, $\PX^0(\mcE) \simeq \mcE$, and thus it is torsion-free by hypothesis. Suppose the assertion is true for $k-1$, and let us conclude that it holds for $k$. From Equation (\ref{Eq: short exact sequence of the sheaf of jets}), we have the natural short exact sequence
    \[
    0 \rightarrow \Sym^k(\Omega_X^1)(\mcE) \rightarrow \PX^k(\mcE) \rightarrow \PX^{k-1}(\mcE) \rightarrow 0.
    \]
    Since both $\Sym^k(\Omega_X^1)(\mcE)$ and $\PX^{k-1}(\mcE)$ are torsion-free, it follows that $\PX^k(\mcE)$ is also torsion-free. 
    
    Therefore, by induction, the proposition holds for all $k\ge 0$.
\end{proof}

\begin{prop}\label{P: jet of reflexive sheaf}
    Let $\mcE$ be a reflexive sheaf over a smooth complex variety $X$. Then,  for all $k\ge 0$, the sheaf $\PX^k(\mcE)$ is also reflexive.
\end{prop}
\begin{proof}
    Since $\mcE$ is reflexive, by \cite[Proposition 1.1]{hartshorne-80-zbMATH03670590} it follows that $\mcE$ can be included in a short exact sequence
    \[
    0 \rightarrow \mcE \rightarrow \mcE' \rightarrow \mcE''\rightarrow 0,
    \]
    where $\mcE'$ is locally free and $\mcE''$ is torsion-free. Because $X$ is smooth, it follows that the functor $\PX^k \otimes $ is exact for all $k\ge 0$ (see \cite[Proposition 16.7.3]{ega-VI-zbMATH03245973}). Hence, applying it to the short exact sequence above, we obtain the short exact sequence
    \[
    0 \rightarrow \PX^k(\mcE) \rightarrow \PX^k(\mcE') \rightarrow \PX^k(\mcE'')\rightarrow 0,
    \]
    with $\PX^k(\mcE')$ locally free and $\PX^k(\mcE'')$ is torsion-free. Therefore, again by \cite[Proposition 1.1]{hartshorne-80-zbMATH03670590}, $\PX^k(\mcE)$ is reflexive. 
\end{proof}

Let $\mcE$ be a coherent sheaf on a complex smooth variety $X$. Dualizing the isomorphism from Equation (\ref{Eq: jets and differential operators}) with respect to the respective canonical structures, we obtain the natural isomorphism
\begin{equation*}
    \begin{split}
        \PX^k(\mcE)^{**} & \rightarrow (\DX{k}(\mcE,\germe_X))^* \\
        \psi & \mapsto \Big(D\in \DX{k}(\mcE,\germe_X) \mapsto \psi(\phi_D)\Big).
    \end{split}
\end{equation*}
Composing it with the canonical morphism $\mathrm{ev}: \PX^k(\mcE) \rightarrow (\PX^k(\mcE))^{**}$, we deduce the natural morphism
\begin{equation}\label{Eq: jet is the dual of differential operators}
    \begin{split}
        \PX^k(\mcE) & \rightarrow (\DX{k}(\mcE,\germe_X))^* \\
        \eta & \mapsto \Big( D\in \DX{k}(\mcE,\germe_X) \mapsto \phi_D(\eta) \Big).
    \end{split}
\end{equation}
The preceding propositions have consequences on this morphism. If $\mcE$ is torsion-free, Proposition \ref{P: jet of torsion-free sheaf} implies that $\PX^k(\mcE)$ is torsion-free, which is equivalent to $\mathrm{ev}: \PX^k(\mcE) \rightarrow (\PX^k(\mcE))^{**}$ being injective. Hence, in this case, the morphism of Equation (\ref{Eq: jet is the dual of differential operators}) is injective. Similarly, if $\mcE$ is reflexive, Proposition \ref{P: jet of reflexive sheaf} implies that the morphism of Equation (\ref{Eq: jet is the dual of differential operators}) is an isomorphism. We will make use of this isomorphism later in this work.

\subsection{The sheaf of first jets and meromorphic connections} 
Let $\mcE$ be a torsion-free sheaf over a smooth complex variety $X$. As a particular case of the short exact sequence of Equation (\ref{Eq: short exact sequence of the sheaf of jets}), we can consider the \emph{short exact sequence associated to the sheaf of first jets of $\mcE$}:
\begin{equation}\label{Eq: short exact sequence of the sheaf of first jets}
        0\rightarrow \Omega_X^1(\mcE) \xrightarrow{\iota} \PX^1(\mcE) \xrightarrow{\pi} \mcE \rightarrow 0
\end{equation}
Let $D\ge 0$ be a divisor on $X$. Recall that a meromorphic connection $\nabla: \mcE\rightarrow \OmegaDE$ corresponds to a \emph{left meromorphic splitting with poles on $D$} of the preceding short exact sequence: that is, an $\germe_X$-linear morphism $\phi: \PX^1(\mcE) \rightarrow \OmegaDE$ such that the composition 
\[
\phi \circ \iota: \Omega_X^1(\mcE) \rightarrow \OmegaDE
\]
is the natural inclusion. Indeed, since $\nabla$ is a differential operator of order $\le 1$, there exists a unique $\germe_X$-linear morphism $\phi_{\nabla}: \PX^1(\mcE) \rightarrow \OmegaDE$ such that  $\nabla = \phi_{\nabla} \circ d^1$. A straightforward computation shows that the Leibniz rule of $\nabla$ is equivalent to the condition that $\phi_{\nabla} \circ \iota$ is the natural inclusion.

Instead of considering a left meromorphic splitting of Equation (\ref{Eq: short exact sequence of the sheaf of first jets}), we can equivalently describe a meromorphic connection as a \emph{right meromorphic splitting} of Equation (\ref{Eq: short exact sequence of the sheaf of first jets}), as follows. Consider the natural inclusions $\PX^1(\mcE) \subset \germe_X(D) \otimes \PX^1(\mcE)$ and $\OmegaDE \subset \germe_X(D) \otimes \PX^1(\mcE)$. A meromorphic connection $\nabla: \mcE \rightarrow \OmegaDE$ induces the $\germe_X$-linear morphism
\begin{equation}\label{Eq: right meromorphic splitting 1}
    \sigma_{\nabla} = d^1_{X,\mcE} - \nabla: \mcE \rightarrow \germe_X(D) \otimes \PX^1(\mcE),    
\end{equation}
which in turn naturally induces a $\germe_X$-linear morphism
\begin{equation}\label{Eq: right meromorphic splitting 2}
    \sigma_{\nabla}: \mcE(-D) \rightarrow \PX^1(\mcE).    
\end{equation}
such that $\pi \circ \sigma_{\nabla}: \mcE(-D) \rightarrow \mcE$ is the canonical inclusion. We refer to $\sigma_{\nabla}$ as the \emph{right meromorphic splitting} of Equation (\ref{Eq: short exact sequence of the sheaf of first jets}) associated to $\nabla$. Conversely, starting with a morphism $\sigma: \mcE(-D) \rightarrow \PX^1(\mcE)$ such that the composition $\pi \circ \sigma: \mcE(-D) \rightarrow \mcE$ is the canonical inclusion, one recovers a meromorphic connection $\nabla$ by the same expressions above.

\begin{rmk}
    During this work, we will be consistent with the notation introduced above. That is, we will denote left meromorphic splittings of Equation (\ref{Eq: short exact sequence of the sheaf of first jets}) by $\phi$ or $\phi_{\nabla}$, while right meromorphic splitting will be denoted by $\sigma$ or $\sigma_{\nabla}$. 
\end{rmk}

Finally, it is useful to describe the relation between meromorphic connections and the meromorphic splitting in terms of local coordinates. Recall that every element $\xi \in \PX^1(\mcE)$ can be written as
\[
\xi = 1\otimes s + \sum \omega_i \otimes s_i,
\]
for some $s\in \mcE, s_i \in \mcE$ and  $\omega_i \in \Omega_X^1$. Given $s\in \mcE(-D)$, it follows directly from Equation (\ref{Eq: right meromorphic splitting 1}) that
\begin{equation}\label{Eq: right meromorphic splitting 3}
      \sigma_{\nabla}(s) = 1\otimes s - \sum \omega_i \otimes s_i \iff \nabla(s) = \sum \omega_i \otimes s_i.
\end{equation}

\subsection{Composition of differential operators}

From the definition of the sheaf of jets, one can verify that, for every $k,l\ge 0$, there exists a well-defined $\germe_X$-linear morphism
\begin{equation}\label{Eq: delta}
    \begin{split}
            \delta: \PX^{k+l}(\mcE) & \rightarrow \PX^l \otimes \PX^k(\mcE) \\
            f \otimes s & \mapsto f \otimes (1 \otimes s),
    \end{split}
\end{equation}
see, for instance, \cite[Lemma 2.3]{berthelot-78-zbMATH03595321} or \cite[Lemma 16.8.9.1]{ega-VI-zbMATH03245973}. The morphism $\delta$ is related to the composition of differential operators in the following sense. Let $D\in \DX{k}(\mcE,\mcE')$ and $D' \in \DX{l}(\mcE',\mcE'')$ be differential operators. Then, we have the following commutative diagram (see \cite[Corollary 2.4]{berthelot-78-zbMATH03595321})
:
\begin{equation*}
    \begin{tikzpicture}
        \matrix(m)[matrix of math nodes, column sep = 2em, row sep = 2em]
        {
        \PX^{k+l}(\mcE) & & \PX^l \otimes \PX^k(\mcE) \\
        & \PX^k(\mcE) & & \PX^l(\mcE') \\
        \mcE & & \mcE' & & \mcE'' \\
        };
        \path[->]
        (m-1-1) edge node[above]{$\delta$} (m-1-3) 
        (m-1-3) edge node[right]{$\mathrm{id} \otimes \phi_D$} (m-2-4)
        (m-2-2) edge node[left]{$d^l_{\PX^k(\mcE)}$}(m-1-3) edge node[above]{$\phi_D$} (m-3-3)
        (m-2-4) edge node[above]{$\phi_{D'}$}(m-3-5)
        (m-3-1) edge node[above]{$D$}(m-3-3) edge node[left]{$d_{\mcE}^{k+l}$} (m-1-1) edge node[above]{$d^k_{\mcE}$} (m-2-2)
        (m-3-3) edge node[above]{$D'$} (m-3-5) edge node[above]{$d^l_{\mcE'}$} (m-2-4)
        ;
    \end{tikzpicture}
\end{equation*}
We define the \emph{$l$-th prolongation} of $D$ as the $\germe_X$-linear morphism
\begin{equation}\label{Eq: prolongation 1}
    \begin{split}
            \phi_D^{l}: \PX^{k+l}(\mcE) & \rightarrow \PX^{l}(\mcE') \\
            f\otimes s & \mapsto f \otimes \phi_D(s)
    \end{split}
\end{equation}
given by the composition of $\delta: \PX^{k+l}(\mcE) \rightarrow \PX^{l} \otimes \PX^k(\mcE)$ and $\rm{id}\otimes \phi_D: \PX^{l} \otimes \PX^k(\mcE) \rightarrow \PX^l(\mcE')$. With this notation, we have that
\begin{equation}\label{Eq: prolongation 2}
    \phi_{D'\circ D} = \phi_{D'} \circ \phi_D^{l}.
\end{equation}

\subsection{Jets and partial connections}

\begin{prop}\label{P: partial connection}
    Let $\F$ be a foliation on a smooth variety $X$, and let $\mcE$ be a coherent sheaf on $X$. Let $\nabla: \mcE \rightarrow \CTF^1(\mcE)$ be a differential operator of order $\le 1$. Then, $\nabla$ is a partial connection if and only if the corresponding $\germe_X$-linear morphism $\phi_{\nabla}: \PX^1(\mcE) \rightarrow \CTF^1(\mcE)$ satisfies that the composition 
    \[
        \phi_{\nabla} \circ \iota: \Omega_X^1(\mcE) \rightarrow \CTF^1(\mcE)
    \]
    is the restriction morphism.
\end{prop}
\begin{proof}
    For any first-order differential operator $\nabla: \mcE \rightarrow \CTF^1(\mcE)$, we have that
    \[
    \phi_{\nabla}(df \otimes s) = \phi_{\nabla}(1\otimes fs - f\otimes s) = \nabla(fs) - f \cdot \nabla(s),  \forall f\in \germe_X,  s\in \mcE.
    \]
    Hence, $\phi_{\nabla} \circ \iota$ coincides with the restriction morphism if, and only if, $\nabla$ satisfies the Leibniz rule.
\end{proof}

Let us consider now some remarks on the prolongation of a partial connection. Let $\nabla: \mcE \rightarrow \CTF^1(\mcE)$ be a partial connection, and let $\phi_{\nabla}: \PX^1(\mcE) \rightarrow \CTF^1(\mcE)$ be the associated $\germe_X$-linear morphism such that $\nabla = \phi_{\nabla} \circ d^1_{X,\mcE}$, as given by isomorphism of Equation (\ref{Eq: jets and differential operators}). For every $k\ge 0$, recall that we define the $k$-th prolongation of $\nabla$ as the morphism
\[
\phi_{\nabla}^k: \PX^{k+1}(\mcE) \rightarrow \PX^k(\CTF^1(\mcE)).
\]
In the same way, for every $v\in T_{\F}$, we can also consider the $k$-th prolongation of the differential operator $\nabla_v$ as the morphism
\[
\phi_{\nabla_v}^k: \PX^{k+1}(\mcE) \rightarrow \PX^k(\CTF^1(\mcE)).
\]
We relate the solutions of the respective prolongations by the following proposition.
\begin{prop}\label{P: prolongation of the connection}
    Let $\F$ be a foliation on a smooth complex variety $X$, and let $(\mcE,\nabla)$ be a partial connection on a coherent sheaf. Then, for all $\xi \in \PX^{k+1}(\mcE)$,
    \[
    \phi_{\nabla}^k(\xi)=0 \iff \phi^k_{\nabla_v}(\xi)=0 \text{ for all } v\in T_{\F}. 
    \]
\end{prop}
\begin{proof}
    From \cite[Equation 16.7.8.2]{ega-VI-zbMATH03245973}, there exists a natural isomorphism
    \begin{equation*}
        \begin{split}
            \PX^k(\CTF^1(\mcE)) & \rightarrow \mcHom_{\PX^k}(\PX^k(T_{\F}), \PX^k(\mcE)) \\
            f \otimes \phi & \mapsto \Big(g\otimes v\in \PX^k(T_{\F}) \mapsto g \cdot f \otimes \phi(v)  \Big),
        \end{split}
    \end{equation*}
    which, by adjunction (see \cite[Corollary 8.9]{AltmanKleiman2021}) with respect to the right $\germe_X$-module structure of $\PX^k(\mcE)$, it induces the isomorphism
    \begin{equation*}
        \begin{split}
        \PX^k(\CTF^1(\mcE)) & \rightarrow \mcHom_{\germe_X}(T_{\F}, \PX^k(\mcE)) \\
        f \otimes \phi & \mapsto \Big(v\in T_{\F} \mapsto f\otimes \phi(v) \Big).
        \end{split}
    \end{equation*}
    Making use of this isomorphism, we regard the prolongation of $\nabla$ as the $\germe_X$-linear morphism
    \begin{equation*}
        \begin{split}
            \phi^k_{\nabla}: \PX^{k+1}(\mcE) & \rightarrow \mcHom_{\germe_X}(T_{\F}, \PX^k(\mcE)). 
        \end{split}
    \end{equation*}
    
    For all $v\in T_{\F}$, considering the evaluation $\mathrm{ev}_v: \mcHom_{\germe_X}(T_{\F}, \PX^k(\mcE)) \rightarrow \PX^k(\mcE)$, we have that
    \[
    \mathrm{ev}_v \circ \phi^k_{\nabla}(f\otimes s) = \mathrm{ev}_v(f\otimes \nabla(s)) = f\cdot \nabla_v(s) = \phi^k_{\nabla_v}(f\otimes s),
    \]
    and therefore, for all $\xi \in \PX^k(\mcE)$,
    \[
        \phi^{k}_{\nabla}(\xi)=0 \iff \phi^{k}_{\nabla_v}(\xi)=0, \forall v\in T_{\F}.
    \]
    This concludes the proof.
\end{proof}

Finally, let us consider the prolongation of partial connection in terms of the natural short exact sequence associated to the sheaf of jets ( as in Equation (\ref{Eq: short exact sequence of the sheaf of jets})). Directly from the definitions, for all $k\ge 1$ we obtain the following commutative diagram:
\begin{equation}\label{D: prolongation of partial connection}
    \begin{tikzpicture}[baseline=(current bounding box.center)]
        \matrix(m)[matrix of math nodes, column sep = 2em, row sep = 2em]
        {
            0 & \Sym^{k+1}(\Omega_X^1)(\mcE) & \PX^{k+1}(\mcE) & \PX^{k}(\mcE) & 0 \\
            0 & \Sym^{k}(\Omega_X^1)\big(\CTF^1(\mcE)\big) & \PX^{k}(\CTF^1(\mcE)) & \PX^{k-1}(\CTF^1(\mcE)) & 0 \\
            };
        \path[->]
            (m-1-1) edge (m-1-2)
            (m-1-2) edge (m-1-3) edge node[left]{$\theta$}(m-2-2)
            (m-1-3) edge (m-1-4) edge node[left]{$\phi_{\nabla}^{k}$}(m-2-3)
            (m-1-4) edge (m-1-5) edge node[left]{$\phi_{\nabla}^{k-1}$}(m-2-4)
            (m-2-1) edge (m-2-2)
            (m-2-2) edge (m-2-3)
            (m-2-3) edge (m-2-4)
            (m-2-4) edge (m-2-5)
        ;
    \end{tikzpicture}
\end{equation}
where $\theta: \Sym^{k+1}(\Omega_X^1)(\mcE) \rightarrow \Sym^{k}(\Omega_X^1)\big(\CTF^1(\mcE)\big)$ is given by 
\[
\theta(df_1\cdots df_{k+1}\otimes e) = \sum_{i=1}^{k+1} df_1 \cdots \widehat{df_i} \cdots df_{k+1} \otimes (\df(f_i) \otimes e).
\]
\begin{prop}
    Let $\F$ be a foliation on a smooth complex variety $X$. Let $\mcE$ be a coherent sheaf on $X$. For all $k\ge 0$, let $\mcHom(\Sym^{k+1} T_X/T_{\F}, \mcE) \rightarrow \Sym^{k+1}(\Omega_X^1)(\mcE)$ be the natural morphism induced by the projection $T_X \rightarrow T_X/T_{\F}$. Then,
    \[
    0 \rightarrow \mcHom(\Sym^{k+1} T_X/T_{\F}, \mcE) \rightarrow \Sym^{k+1}(\Omega_X^1)(\mcE) \xrightarrow{\theta} \Sym^k(\Omega_X^1)(\CTF^1(\mcE))
    \]
    is a short exact sequence.
\end{prop}
\begin{proof}
    Let us start with the short exact sequence determined by the tangent sheaf of the foliation:
    \[
    0 \rightarrow T_{\F} \rightarrow T_X \rightarrow \frac{T_X}{T_{\F}} \rightarrow 0
    \]
    Applying the functor $\Sym^{k+1}$, we deduce the short exact sequence (see \cite[Proposition A2.2]{eisenbud-zbMATH00704831})
    \[
    T_{\F} \otimes \Sym^k(T_X) \rightarrow \Sym^{k+1}(T_X) \rightarrow \Sym^{k+1}\left(\frac{T_X}{T_{\F}} \right) \rightarrow 0.
    \]
    Next, applying the functor $\mcHom(\cdot, \mcE)$, we obtain
    \[
    0 \rightarrow \mcHom(\Sym^{k+1} T_X/T_{\F}, \mcE) \rightarrow \mcHom(\Sym^{k+1} T_X, \mcE) \rightarrow \mcHom(T_{\F} \otimes \Sym^{k} T_X, \mcE)
    \]
    Finally, using the natural isomorphisms $\Sym^{k+1}(\Omega_X^1) \otimes \mcE \rightarrow \mcHom(\Sym^{k+1} T_X, \mcE)$ and $\Sym^k(\Omega_X^1) \otimes \CTF^1(\mcE) \rightarrow \mcHom(\Sym^k T_X, \CTF^1(\mcE))$, a straightforward verification shows that we obtain the morphism $\theta$ in the short exact sequence. This concludes the proof.
\end{proof}

Observe that we have a natural isomorphism 
\[
    \mcHom(\Sym^{k+1}T_X/T_{\F},\mcE) \simeq \mcHom(\Sym^{k+1} N_{\F}, \mcE),
\]
and for consistency with the notation we introduced up to now, let us denote these sheaves by $\Sym^{k+1}(\CNF)(\mcE)$. Therefore, from Diagram (\ref{D: prolongation of partial connection}), we obtain the short exact sequence
\begin{equation}\label{Eq: short exact sequence of the sheaf of transverse jets 0}
    0 \rightarrow \Sym^{k+1}(\CNF)(\mcE) \rightarrow \ker (\phi^{k+1}_{\nabla}) \rightarrow \ker(\phi^{k}_{\nabla})
\end{equation}

\subsection{Transverse jets of flat partial connections}\label{Subsection: transverse jets}

Let $\F$ be a foliation on a complex smooth variety $X$. Let $(\mcE,\nabla)$ be a flat partial connection on a coherent sheaf. For every $k\ge 0$, we define the \emph{$k$-th sheaf of transverse jets of $(\mcE,\nabla)$} as
\[
\PXF^k(\nabla):= \big(\DXF{k}(\nabla,\germe_X)\big)^*.
\]
Considering the flat partial connection $(\nabla,\df)^k$ on $\DXF{k}(\nabla,\germe_X)$ (see Theorem \ref{T: connection on the transverse differential operators}), the sheaf of transverse jets is naturally endowed with a flat partial connection given as the dual of $(\nabla,\df)^k$, which we denoted by $\nabla^k$, and it is given by
\begin{equation}\label{Eq: definition of the connection on transverse jets}
    \nabla_v^k(\xi)(D) = v(\xi(D)) - \xi((\nabla,\df)^k_v(D)),
\end{equation}
for all $v\in T_{\F}$,$\xi \in \PXF^k(\nabla)$, and $D\in \DXF{k}(\nabla,\germe_X)$. We refer to the pair $(\PXF^k(\nabla), \nabla^k)$ as the \emph{$k$-th jet of the flat partial connection $(\mcE,\nabla)$}.

Let us leave aside the connection $\nabla^k$ on $\PXF^k(\nabla)$ for a moment, and focus only on the sheaf of transverse jets. Recall that we have a short exact sequence obtained by the definition of the sheaf of transverse parts:
\[
0 \rightarrow S^k(\nabla) \rightarrow \DX{k}(\mcE,\germe_X) \xrightarrow{\pi} \DXF{k}(\nabla,\germe_X) \rightarrow 0.
\]
Dualizing, it induces a short exact sequence 
\begin{equation}\label{Eq: definition of the sheaf of transverse jets 1}
    0 \rightarrow \PXF^k(\nabla) \xrightarrow{\iota} \big(\DX{k}(\mcE,\germe_X)\big)^* \rightarrow \big(S^k(\nabla)\big)^*,
\end{equation}
where for every $\xi \in \PXF^k(\nabla)$ and $D\in \DX{k}(\mcE,\germe_X)$, $\iota(\xi)(D) = \xi(\pi(D))$.

Although the discussion above does not require any further assumption on $\mcE$, from now on we are interested in the case where $\mcE$ is reflexive. Recall that, in this case, we have that $\PX^k(\mcE)$ is reflexive by Proposition \ref{P: jet of reflexive sheaf}, and as in Equation (\ref{Eq: jet is the dual of differential operators}), we have that the natural morphism
\begin{equation*}
    \begin{split}
        \PX^k(\mcE) & \rightarrow \big(\DX{k}(\mcE,\germe_X)\big)^* \\
        \xi & \mapsto \Big(D\in \DX{k}(\mcE,\germe_X) \mapsto \phi_D(\xi) \Big)
    \end{split}
\end{equation*}
is an isomorphism. Combined with the short exact sequence of Equation (\ref{Eq: definition of the sheaf of transverse jets 1}), this induces the short exact sequence
\begin{equation}\label{Eq: definition of the sheaf of transverse jets 2}
    0 \rightarrow \PXF^k(\nabla) \xrightarrow{\iota} \PX^k(\mcE) \rightarrow (S^k(\nabla))^*
\end{equation}
such that $\phi_D(\iota(\xi)) = \xi(\pi(D))$ for all $\xi \in \PXF^k(\nabla)$ and $D\in \DX{k}(\mcE,\germe_X)$. This allows us to regard $\PXF^k(\nabla)$ as a $\germe_X$-submodule of $\PX^k(\mcE)$ (with respect to the canonical structure). We abuse notation and denote $\iota(\xi)$ simply by $\xi$.

Concerning the morphism $\PX^k(\mcE) \rightarrow \big(S^k(\nabla)\big)^*$, it corresponds to the evaluation of the system of differential operators $S^k(\nabla)$ on the elements of $\PX^k(\mcE)$. Hence, we have that

\begin{equation*}
    \begin{split}
    \xi \in \PXF^k(\nabla) & \iff \phi_{D \circ \nabla_v}(\xi)=0, \forall D\in \DX{k-1}(\mcE,\germe_X), v\in T_{\F} \\
    & \xLeftrightarrow{\text{Eq.~(\ref{Eq: prolongation 2})}} \phi_D \circ \phi_{\nabla_v}^{k-1}(\xi) =0, \forall D\in \DX{k-1}(\mcE,\germe_X), v\in T_{\F} \\
    & \xLeftrightarrow{\text{Prop.~\ref{P: jet of torsion-free sheaf}}} \phi_{\nabla_v}^{k-1}(\xi)=0, \forall v\in T_{\F}
\end{split}
\end{equation*}
Thus, by Proposition \ref{P: prolongation of the connection}, we obtain from Equation (\ref{Eq: definition of the sheaf of transverse jets 2}) the short exact sequence
\begin{equation}\label{Eq: definition of the sheaf of transverse jets 3}
    0 \rightarrow \PXF^k(\nabla) \xrightarrow{\iota} \PX^k(\mcE) \xrightarrow{\phi_{\nabla}^{k-1}} \PX^{k-1}(\CTF^1(\mcE)),
\end{equation}
Finally, given a section $\xi = \sum f_i \otimes s_i \in \PX^k(\mcE)$, we deduce from Equation (\ref{Eq: prolongation 2}) that
    \begin{equation}\label{Eq: definition of the sheaf of transverse jets 4}
            \xi \in \PXF^k(\nabla) \iff \sum f_i \otimes \nabla(s_i)=0 \in \PX^{k-1}(\CTF^1(\mcE)).
    \end{equation}

After this discussion on the sheaf of transverse jets, let us now bring back the connection $\nabla^k$ on $\PXF^k(\nabla)$.
\begin{prop}\label{P: description of the partial connection on the sheaf of transverse jets}
     Let $\F$ be a foliation on a smooth complex variety $X$. Let $(\mcE,\nabla)$ be a flat partial connection on a reflexive sheaf. Let $\xi = \sum f_i \otimes s_i \in \PXF^k(\nabla)$ be a section of the sheaf of transverse jets. Then, for all $v\in T_{\F}$,
    \begin{equation*}
        \nabla_v^k(\xi) = \sum  v(f_i) \otimes s_i + f_i \otimes \nabla_v(s_i).
    \end{equation*}
\end{prop}
\begin{proof}
    First, a straightforward calculation using  Equation (\ref{Eq: definition of the sheaf of transverse jets 4}) shows that
    \[
        \sum  v(f_i) \otimes s_i + f_i \otimes \nabla_v(s_i) \in \PXF^k(\nabla).
    \]
    Using the expression for the $\nabla_v^k$ as in Equation (\ref{Eq: definition of the connection on transverse jets}), for all $D\in \DXF{k}(\nabla,\mcE)$ we calculate that
    \begin{equation*}
        \begin{split}
            \nabla^k_v(\xi)(D) & = v\big(\xi(D)\big) - \xi\big((\nabla,\df)^k_v(D)\big) \\
            & = v\big(\phi_{\widehat{D}}(\xi)\big) - \phi_{v\circ \widehat{D} - \widehat{D} \circ \nabla_v}(\xi) \\
            & = v\left(\sum f_i \cdot \widehat{D}(s_i) \right) - \sum f_i \cdot \big(v\circ \widehat{D} - \widehat{D} \circ \nabla_v\big)(s_i) \\
            & = \sum v(f_i)\cdot \widehat{D}(s_i) + \sum f_i \cdot \widehat{D}\big(\nabla_v(s_i)\big) \\
            & = \phi_{\widehat{D}}\left(\sum v(f_i)\otimes s_i + f_i \otimes \nabla_v(s_i)\right) \\
            & = \left(\sum v(f_i)\otimes s_i + f_i \otimes \nabla_v(s_i)\right)(D),
        \end{split}
    \end{equation*}
    where $\widehat{D}\in \DX{k}(\mcE,\germe_X)$ is any differential operator such that $\pi(\widehat{D})=D$. Since the equality holds for all $D$, it follows that $\nabla_v^k(\xi) = \sum  v(f_i) \otimes s_i + f_i \otimes \nabla_v(s_i)$. This concludes the proof.
\end{proof}

This is the appropriate moment to recall that in \cite[Section 4.2]{fazoli25-arXiv:2505.11662}, under the additional hypothesis of smoothness on $\F$, one can find another construction for the $k$-jet of a flat partial connection on a locally free sheaf. This construction, as discussed in Section \ref{Subsection: introduction, partial connections, jets}, relies on the correspondence between flat partial connections on locally free sheaves, and locally free $\germe_{X/\F}$-modules (see \cite[Corollary 3.8]{fazoli25-arXiv:2505.11662}), which in turn works only on the analytic category. 

On the other hand, the preceding discussion applies in a much more general context: it works both in analytic and algebraic categories, for singular foliations, and for reflexive sheaves. Hence, in order to complete the discussion we presented in the introduction, it remains to verify that both constructions coincide in the context of \cite[Section 4.2]{fazoli25-arXiv:2505.11662}.

\begin{cor}\label{Cor: transverse jets}
    Let $\F$ be a smooth foliation on a smooth complex variety $X$ in the analytic category. Let $(\mcE,\nabla)$ be a flat partial connection on a locally free sheaf. Then
    \[
    \PXF^k(\nabla) = \left\{\sum f_i \otimes s_i \mid f_i\in \germe_X, s_i \in \ker \nabla \right\} \subset \PX^k(\mcE),
    \]
    and
    \[
    \ker \nabla^k = \left\{\sum f_i \otimes s_i \mid f_i\in \germe_{X/\F}, s_i \in \ker \nabla \right\} \subset \PX^k(\mcE).
    \]
    That is, the flat partial connection $(\PXF^k(\nabla),\nabla^k)$ coincides with the $k$-jet of the flat partial connection $(\mcE,\nabla)$ as defined in \cite[Section 4.2]{fazoli25-arXiv:2505.11662}.
\end{cor}
\begin{proof}
    The proof is a straightforward computation using Equation (\ref{Eq: definition of the sheaf of transverse jets 4}) and Proposition \ref{P: description of the partial connection on the sheaf of transverse jets} with a local flat basis for $\mcE$.
\end{proof}

\subsection{The short exact sequence of the sheaves of transverse jets}

Let $(\mcE,\nabla)$ be a flat partial connection on a reflexive sheaf. Using the characterization of the sheaf of transverse jets of Equation (\ref{Eq: definition of the sheaf of transverse jets 3}), the short exact sequence of Equation (\ref{Eq: short exact sequence of the sheaf of transverse jets 0}) becomes
\begin{equation}\label{Eq: short exact sequence of the sheaf of transverse jets}
    0 \rightarrow \Sym^k(\CNF)(\mcE) \xrightarrow{\iota} \PXF^k(\nabla) \xrightarrow{\pi} \PXF^{k-1}(\nabla).
\end{equation}
We refer to it as the \emph{short exact sequence associated to the $k$-th sheaf of transverse jets of $\nabla$}.

\begin{rmk}
    Contrary to the case of smooth foliations and transverse jets of flat partial connections on locally free sheaves (as in \cite[Corollary 4.6]{fazoli25-arXiv:2505.11662}), the morphism $\PXF^k(\nabla) \rightarrow \PXF^{k-1}(\nabla)$ is not necessarily surjective. For instance, for $(\mcE,\nabla)=(\CNF,\nabla_B)$, the image of $\PXF^1(\nabla_B) \rightarrow \CNF$ is related to the concept of the ideal of persistent singularities. This will be discussed in Section \ref{Subsection: persistent singularities}. 
\end{rmk}

Observe that the sheaves involved in Equation (\ref{Eq: short exact sequence of the sheaf of transverse jets}) are naturally endowed with flat partial connections: each $\PXF^k(\nabla)$ is endowed with $\nabla^k$ as in Proposition \ref{P: description of the partial connection on the sheaf of transverse jets}, while $\Sym^k(\CNF)(\mcE)$ is endowed with the flat partial connection  $\nabla' := \mcHom(\Sym^k \nabla_B, \nabla)$ given by
\begin{equation}\label{Eq: flat partial connection in the symbol 0}
    \nabla_v'(\phi)(w_1,\cdots,w_k) = \nabla_v(\phi(w_1,\cdots,w_k)) - \sum_{i=1}^k\phi \Big(w_1, \cdots ,(\nabla_B)_v w_i, \cdots w_k \Big), 
\end{equation}
for all $v\in T_{\F}$, $w_i \in T_X/T_{\F}$ and $\phi \in  \Sym^k(\CNF)(\mcE)$. This is the natural partial connection on $\Sym^k(\CNF)(\mcE)$ considering $\Sym^k$ and $\mcHom$ in the category of flat partial connections on coherent sheaves  (see \cite[Section 3.4]{fazoli25-arXiv:2505.11662}). Moreover, for elements $\Omega \otimes s \in \Sym^k(\CNF)(\mcE)$ with $\Omega \in \Sym^k(\CNF)$ and $s\in \mcE$, a straightforward calculation using Equation (\ref{Eq: flat partial connection in the symbol 0}) shows that
\begin{equation}\label{Eq: flat partial connection in the symbol 1}
    \nabla_v'(\Omega \otimes s) = \Lie_v(\Omega)\otimes s + \Omega \otimes \nabla_v(s). 
\end{equation}

\begin{prop}\label{P: short exact sequence of the sheaf of transverse jets} Let $\F$ be a foliation on a smooth complex variety $X$. Let $(\mcE,\nabla)$ be a flat partial connection on a reflexive sheaf. Then, for all $k\ge 1$, 
    \[
        0 \rightarrow \Big(\Sym^k(\CNF)(\mcE), \nabla' \Big) \xrightarrow{\iota}\big(\PXF^k(\nabla),\nabla^k\big) \xrightarrow{\pi} \big(\PXF^{k-1}(\nabla),\nabla^{k-1}\big)
    \]
is a short exact sequence of flat partial connections, where $\nabla'$ is the connection given by Equation (\ref{Eq: flat partial connection in the symbol 0}).
\end{prop}
\begin{proof}
    We need to verify that $\iota$ and $\pi$ are horizontal morphisms. The fact that $\pi$ is horizontal follows directly from the expressions from $\nabla^k$ and $\nabla^{k-1}$ given by Proposition \ref{P: description of the partial connection on the sheaf of transverse jets}. Hence, it remains to verify that $\iota$ is horizontal as well.

    Observe that is enough to verify that $\iota$ is horizontal only for the sections in the image of the natural morphism $\Sym^k(\CNF) \otimes \mcE \rightarrow \Sym^k(\CNF)(\mcE)$. Let $\Omega \otimes s$ be an element of $\Sym^k(\CNF) \otimes \mcE$, which by the morphism $\iota$, it can be seen as an element of $\Sym^k(\Omega_X) \otimes \mcE$. Recall that, since $\mcE$ is reflexive, we have the isomorphism of Equation (\ref{Eq: jet is the dual of differential operators}), which allow us to see $\Omega\otimes s$ as a morphism
    \begin{equation*}
        \begin{split}
                \Omega\otimes s: \DX{k}(\mcE,\germe_X) & \rightarrow \germe_X \\
                D & \mapsto \phi_D(\Omega\otimes s).
        \end{split}
    \end{equation*}
    Since $\Omega\otimes s \in \ker (\pi: \PX^k(\mcE)\rightarrow \PX^{k-1}(\mcE))$, it follows that $\Omega\otimes s$ vanishes on $\DX{k-1}(\mcE,\germe_X)$. Moreover, for any $D = w_1\circ \cdots \circ w_k \otimes \phi$, where $w_i\in T_X$ and $\phi \in \mcE^*$, we have that
    \[
    \Omega\otimes s(D) = \Omega(w_1,\ldots,w_k) \cdot \phi(s).
    \]
    Hence, using the expression for $\nabla^k$ as in Equation (\ref{Eq: definition of the connection on transverse jets}), we have that
    \[
    \nabla^k(\Omega\otimes s)(D) = v \Big(\Omega(w_1,\ldots,w_k) \cdot \phi(s)\Big) - \Omega\otimes s\Big( (\nabla,\df)_v^k(D)\Big).
    \]
    Using properties of the Lie derivative and Remark \ref{Rmk: expression for the fpc on transverse diff ope}, a straightforward calculation shows that
    \[
    \nabla^k(\Omega\otimes s)(D) = \Lie_v(\Omega)(w_1,\ldots,w_k) \cdot \phi(s) + \Omega(w_1,\ldots,w_k) \cdot \phi(\nabla_v(s)).
    \]
    Thus,
    \[
    \nabla_v^k(\Omega \otimes s) = \Lie_v(\Omega) \otimes s + \Omega \otimes \nabla_v(s),
    \]
    and therefore, by Equation (\ref{Eq: flat partial connection in the symbol 1}), we conclude that $\nabla^k(\Omega\otimes s) = \nabla'(\Omega \otimes s)$. This concludes the proof.
\end{proof}

\section{The first transverse jet and extensions of meromorphic partial connections}\label{Section: extensions}

\subsection{The first transverse jet}\label{Subsection: the first transverse jet}

Let $\F$ be a foliation on a smooth variety $X$. Let $(\mcE,\nabla)$ be a flat partial connection on a reflexive sheaf, and let us consider the first sheaf of transverse jets of $\nabla$. Recall that, by Equation (\ref{Eq: short exact sequence of the sheaf of transverse jets}), $\PXF^1(\nabla)$ fits in the short exact sequence
\begin{equation}\label{Eq: short exact sequence of the first sheaf of transverse jets}
    0 \rightarrow \CNF(\mcE) \xrightarrow{\iota} \PXF^1(\nabla) \xrightarrow{\pi} \mcE,
\end{equation}
where $\iota$ is the morphism induced by the natural inclusion $\Omega_X^1(\mcE) \rightarrow \PX^1(\mcE)$, and $\pi$ is induced by the natural projection $\pi: \PX^1(\mcE) \rightarrow \mcE$. Moreover, by Proposition \ref{P: short exact sequence of the sheaf of transverse jets}, considering the respective natural flat partial connections, this is a short exact sequence in the category of flat partial connections.

Recall that every section $\xi \in \PX^1(\mcE)$ can be written in the form
\[
\xi = 1\otimes s + \sum \omega_i\otimes s_i, 
\]
for some $s\in \mcE$, $s_i\in \mcE$ and $\omega_i\in \Omega_X^1$. Regarding this expression, Equation (\ref{Eq: definition of the sheaf of transverse jets 4}) and Proposition \ref{P: description of the partial connection on the sheaf of transverse jets} admit the following reformulations.

\begin{prop}
    Let $\F$ be a foliation on a smooth variety $X$, and let $(\mcE,\nabla)$ be a flat partial connection on a reflexive sheaf. Then,
    \begin{equation}\label{Eq: description of the first sheaf of transverse jets}
        \PXF^1(\nabla) = \left\{1\otimes s - \sum \omega_i \otimes s_i \mid \nabla(s) = \sum \restr{\omega_i}{T_{\F}} \otimes s_i\right\} \subset \PX^1(\mcE).
    \end{equation}
    Moreover, for every $\xi = 1\otimes s - \sum \omega_i \otimes s_i \in \PXF^1(\nabla)$, we have that
    \begin{equation}\label{Eq: description of the partial connection on the first sheaf of transverse jets}
        \nabla^1_v(\xi) = 1\otimes \nabla_v(s) - \sum \omega_i \otimes \nabla_v(s_i) - \Lie_v(\omega_i) \otimes s_i
    \end{equation}
\end{prop}
\begin{proof}
    Given 
    \[
        \xi = 1\otimes s - \sum \omega_i \otimes s_i \in \PX^1(\mcE),
    \]
    we have by Proposition \ref{P: partial connection} that
    \[
    \phi_{\nabla}(\xi) = \nabla(s) - \sum \restr{\omega_i}{T_{\F}} \otimes s_i,
    \]
    Hence, from Equation (\ref{Eq: definition of the sheaf of transverse jets 3}), we have the 
    \[
        \xi \in \PXF^1(\nabla) \iff \nabla(s) = \sum \restr{\omega_i}{T_{\F}} \otimes s_i,
    \]
    which is exactly the claim of Equation (\ref{Eq: description of the first sheaf of transverse jets}).

    To compute $\nabla_v(\xi)$ from this expression, note that there is no loss of generalization in supposing that $\omega_i  = f_idg_i$ with $f_i,g_i \in \germe_X$. By rewriting $\omega_i \otimes s_i$ as $f_i\otimes g_is_i - f_ig_i \otimes s_i$, Equation (\ref{Eq: description of the partial connection on the first sheaf of transverse jets}) then follows from a straightforward calculation using Proposition \ref{P: description of the partial connection on the sheaf of transverse jets}.
\end{proof}

\begin{cor}\label{Cor: image in the short exact sequence of transverse jets} Let $\F$ be a foliation on a smooth variety $X$, and let $(\mcE,\nabla)$ be a flat partial connection on a reflexive sheaf. Let $\pi: \PXF^1(\nabla) \rightarrow \mcE$ be the natural projection as in Equation (\ref{Eq: short exact sequence of the first sheaf of transverse jets}). Then,
    \[
        \mathrm{im}(\pi) = \left\{ s \in \mcE \mid  \exists \omega_i \in \Omega_X^1, s_i \in \mcE \text{ such that } \nabla(s) = \sum \restr{\omega_i}{T_{\F}} \otimes s_i \right\}.
    \]
\end{cor}

\subsection{The first transverse jet of the Bott connection}\label{Subsection: persistent singularities}

Let $\F$ be a foliation on a smooth complex variety $X$. Let us specify the setting of the preceding section to the first sheaf of the transverse jets of the Bott connection. Although we focus here on the Bott connection on the conormal sheaf, the discussion also applies to the Bott connection on the normal sheaf.

Recall that the \emph{ideal of persistent singularities of $\F$} is defined as
\[
    \idealpers(\F) = \Big\{f\in \germe_X \mid \forall \omega \in \CNF, \exists \eta_i\in \Omega_X^1, \omega_i \in \CNF ; f\cdot d\omega = \sum \eta_i \wedge \omega_i \Big\} \subset \germe_X.
\]
If $\idealpers(\F)\neq \germe_X$, the subscheme defined by $\idealpers(\F)$ is called the \emph{subscheme of persistent singularities} of $\F$. Otherwise, we say that $\F$ is a foliation \emph{without persistent singularities}. See \cite[Section 4]{massri-molinuevo-quallbrunn-21-zbMATH07317130} for this nomenclature and for a more detailed discussion. The ideal of persistent singularities is related to the first sheaf of transverse jets of the Bott connection by the following proposition.

\begin{prop}\label{P: ideal of persistent singularities}  Let $\F$ be a foliation on a smooth complex variety $X$. Let $(\CNF,\nabla_B)$ be the Bott connection on the conormal sheaf of $\F$, and let $\pi: \PXF^1(\nabla_B) \rightarrow \CNF$ the natural projection appearing in the short exact sequence associated to $\PXF^1(\nabla_B)$. Then,
\[
\idealpers(\F) \otimes \CNF \subset \rm{im}(\pi: \PXF^1(\nabla_B) \rightarrow \CNF) \subset \CNF.
\]
\end{prop}
\begin{proof}
    For every $f\in \idealpers(\F)$ and $\omega \in \CNF$, by definition there exists a collection $\eta_i\in \Omega_X^1, \omega_i \in \CNF$ such that
    \[
    f \cdot d\omega = \sum \eta_i \wedge \omega_i,
    \]
    which, by the definition of the Bott connection, implies that
    \[
    f \cdot \nabla_v(\omega) =  f \cdot \Lie_v(\omega) = \iota_v\left(\sum \eta_i \wedge \omega_i\right) = \sum  \iota_v(\eta_i) \cdot \omega_i, \forall v\in T_{\F}.
    \]
    Hence, $f\cdot \nabla(\omega) = \sum \restr{\eta_i}{T_{\F}} \otimes \omega_i \in \CTF^1\otimes \CNF$, and by Proposition \ref{Eq: description of the first sheaf of transverse jets}, it follows that
    \[
    -\sum \eta_i \otimes \omega_i + f \otimes \omega \in \PXF^1(\nabla_B). 
    \]
    Therefore, $f\cdot \omega \in \rm{im}(\pi: \PXF^1(\nabla_B) \rightarrow \CNF)$. This concludes the proof.
\end{proof}

For codimension one foliations, we have indeed the following.

\begin{prop}\label{P: codimension one, ideal of persistent singularities} Let $\F$ be a codimension one foliation on a smooth complex variety $X$. Let $(\CNF, \nabla_B)$ be the Bott connection on the conormal sheaf of $\F$. Then,
\[
 \rm{im}(\pi: \PXF^1(\nabla_B) \rightarrow \CNF) = \idealpers(\F) \otimes \CNF.
\]
\end{prop}
\begin{proof}
    Since the inclusion $\idealpers(\F) \otimes \CNF\subset \rm{im}(\pi)$ was stablished in Proposition \ref{P: ideal of persistent singularities}, it remains to verify the reverse inclusion. 

    Starting with the short exact sequence induced by the inclusion $T_{\F}\subset T_X$, and applying the functor $\bigwedge^2$, to it, we obtain by \cite[Proposition A2.2]{eisenbud-zbMATH00704831} the short exact sequence
    \[
        T_{\F} \otimes T_X \rightarrow \bigwedge^2 T_X \rightarrow \bigwedge^2 (T_X/T_{\F}) \rightarrow 0. 
    \]
    Dualizing, it induces the short exact sequence
    \[
        \left(\bigwedge^2 (T_X/T_{\F}) \right)^* \rightarrow \Omega_X^2\rightarrow \mcHom(T_{\F}\otimes T_X,\germe_X),
    \]
    By adjunction, we have a natural isomorphism $\mcHom(T_{\F}\otimes T_X,\germe_X)\rightarrow \mcHom(T_{\F}, \Omega_X^1)$, and under this isomorphism, the morphism $\Omega_X^2\rightarrow \mcHom(T_{\F}\otimes T_X,\germe_X)$ turns out to be 
    \begin{equation*}
        \begin{split}
            \iota_{\bullet}: \Omega_X^2 & \rightarrow \mcHom(T_{\F}, \Omega_X^1) \\
            \omega & \mapsto (v\in T_{\F} \mapsto i_v(\omega)).
        \end{split}
    \end{equation*}
    Moreover, since $\F$ has codimension one, we have $\left(\bigwedge^2 (T_X/T_{\F}) \right)^* \simeq 0$. Therefore, in this case, the natural morphism $\iota_{\bullet}$ is injective. Let us make use of this injectivity to conclude the proof.
    
    Let $\omega\in \CNF$ be a local generator of $\CNF$, and let $f\in \germe_X$ such that $f\cdot \omega \in \rm{im}(\pi)$. By Corollary \ref{Cor: image in the short exact sequence of transverse jets}, there exists $\eta \in \Omega_X^1$ such that
    \[
    f \cdot \nabla_B(\omega) = \restr{\eta}{T_{\F}} \otimes \omega \in \mcHom(T_{\F},\CNF).
    \]
    Since $(\nabla_B)_v(\omega) = \Lie_v(\omega) = \iota_v(d\omega)$ by the definition of the Bott connection, we have that
    \[
    \iota_{\bullet}(f \cdot d\omega) = \iota_{\bullet}(\eta \wedge \omega).
    \]
    Because $\iota_{\bullet}$ is injective, it follows that $f\cdot \omega = \eta \wedge \omega$, that is, $f\cdot \omega \in \idealpers(\F) \otimes \CNF$. Therefore, 
    \[
    \rm{im}(\pi) \subset \idealpers(\F) \otimes \CNF,
    \]
    and this concludes the proof of the proposition.
    \end{proof}

\subsection{Extensions}

\begin{prop}\label{P: extensions and splittings}
    Let $\F$ be a foliation on a smooth complex variety $X$. Let $(\mcE,\nabla)$ be a flat partial connection on a reflexive sheaf. Let $\hatnabla: \mcE \rightarrow \OmegaDE$ be a meromorphic connection. Then, $\hatnabla$ extends $\nabla$ if, and only if, the corresponding meromorphic splitting
    \[
        \sigma: \mcE(-D) \rightarrow \PX^1(\mcE) 
    \]
    induced by Equation (\ref{Eq: right meromorphic splitting 2})
    factors through the inclusion $\PXF^1(\nabla) \subset \PX^1(\mcE)$.
\end{prop}
\begin{proof}
    By definition, $\sigma$ factors through the inclusion $\PXF^1(\nabla) \subset \PX^1(\mcE)$ if and only if the composition $\phi_{\nabla} \circ \sigma: \mcE(-D) \rightarrow \CTF^1(\mcE)$ is vanishing. Let us consider the meromorphic connection $\hatnabla$ as a $\C$-linear morphism $\hatnabla: \mcE(-D) \rightarrow \Omega_X^1(\mcE)$. For every $s\in \mcE(-D)$, we have that
    \[
    \phi_{\nabla} \circ \sigma(s) = \phi_{\nabla} \circ (d^1_{X,\mcE} - \hatnabla)(s) = \nabla(s) - \rm{restr} \circ \hatnabla(s),
    \]
    where recall that $\mathrm{restr}: \Omega_X^1(\mcE) \rightarrow \CTF^1(\mcE)$ is the morphism induced by the inclusion $T_{\F} \rightarrow T_X$.
    Thus, $\phi_{\nabla} \circ \sigma \equiv 0$ if and only if we have the commutative diagram:
    \begin{equation*}
        \begin{tikzpicture}
            \matrix(m)[matrix of math nodes, column sep = 2em, row sep = 2em]
            {
            \mcE(-D) & \Omega_X^1(\mcE) \\
            \mcE & \CTF(\mcE) \\
            };
            \path[->]
            (m-1-1) edge node[above]{$\hatnabla$} (m-1-2) edge node[left]{$\iota$} (m-2-1)
            (m-1-2) edge node[left]{$\mathrm{restr}$} (m-2-2)
            (m-2-1) edge node[above]{$\nabla$} (m-2-2)
            ;
        \end{tikzpicture}
    \end{equation*}
    Since $\CTF^1(\mcE)$ is torsion-free, by using the Leibniz rule we can conclude that $\nabla(s) = \rm{restr} \circ \hatnabla(s)$ for all $s\in \mcE$, that is, $\hatnabla$ extends $\nabla$.
\end{proof}

As a consequence of Proposition \ref{P: extensions and splittings}, each meromorphic extensions $\hatnabla$ of a flat partial connection on a reflexive sheaf $(\mcE,\nabla)$ induces a meromorphic splitting
\[
\sigma_{\hatnabla}: \mcE(-D) \rightarrow \PXF^1(\nabla)
\]
of the natural short exact sequence associated to $\PXF^1(\nabla)$ (see Equation (\ref{Eq: short exact sequence of the first sheaf of transverse jets})).

\begin{thm}\label{T: meromorphic extensions} Let $\F$ be a foliation on a complex smooth variety $X$. Let $(\mcE,\nabla)$ be a flat partial connection on a reflexive sheaf. Let $D\ge 0$ be a divisor on $X$. Then, there exists a natural correspondence between:
\begin{enumerate}[label = (\alph*)]
    \item\label{Item: meromorphic extension} meromorphic connections $\hatnabla: \mcE \rightarrow \OmegaDE$ extending $\nabla$; and
    \item\label{Item: meromorphic splitting} meromorphic splittings with poles on $D$ of the short exact sequence associated to the first sheaf of transverse jets of $\nabla$;
\end{enumerate}
    where the correspondence is given by taking a meromorphic connection $\hatnabla$ extending $\nabla$ to the meromorphic splitting $ \sigma_{\hatnabla}: \mcE(-D) \rightarrow  \PXF^1(\nabla)$ given by Proposition \ref{P: extensions and splittings}.
\end{thm}
\begin{proof}
    We have already established in Proposition \ref{P: extensions and splittings} that a meromorphic extension of $\nabla$ induces the meromorphic splitting $\sigma_{\hatnabla}$. It remains to verify the opposite direction. 
    
    Let $\sigma: \mcE(-D) \rightarrow \PXF^1(\nabla)$ be a meromorphic splitting of short exact sequence associated to $\PXF^1(\nabla)$. Composing with the inclusion $\PXF^1(\nabla)\subset \PX^1(\mcE)$, $\sigma$ induces a meromorphic splitting of the short exact sequence associated to the sheaf of first jets of $\mcE$, which corresponds to a meromorphic connection $\hatnabla: \mcE\rightarrow \OmegaDE$ as in Equation (\ref{Eq: short exact sequence of the sheaf of first jets}). Again by Proposition \ref{P: extensions and splittings}, $\hatnabla$ extends $\nabla$. Moreover, it follows directly by definition that $\sigma = \sigma_{\hatnabla}$, and therefore we have established the correspondence.
\end{proof}

\begin{rmk}\label{Rmk: meromorphic splitting when we have line bundles}
    When $\mcE = \mc L$, there is no need to fix the divisor $D$ \emph{a priori} in order to apply Theorem \ref{T: meromorphic extensions}, in the following sense. Let $\mcL'$ be a line bundle and let $\phi:\mc L'\rightarrow \PXF^1(\nabla)$ be an $\germe_X$-linear morphism. Denote by $\pi: \PXF^1(\nabla)\rightarrow \mcE$ the natural projection. Then, the composition $\pi \circ \phi: \mc L' \rightarrow \mc L$ is either the zero morphism or injective. In the injective case, $\mc L'$ is isomorphic to a sub-line bundle of $\mc L$, that is, $\mc L'\simeq \mc L(-D)$ for some effective divisor $D$. Therefore, $\phi$ induces a meromorphic splitting with poles on $D$ of the short exact sequence associated to $\PXF^1(\nabla)$, and then Theorem \ref{T: meromorphic extensions} applies.
\end{rmk}

In the spirit of \cite[Theorem 4.11]{massri-molinuevo-quallbrunn-21-zbMATH07317130}, we deduce some interesting corollaries.
    
\begin{cor} Let $X$ be a compact Kahler manifold, and let $\mcE$ be a reflexive sheaf on $X$. Suppose there exists a foliation $\F$ on $X$ and a $\F$-partial connection on $\mcE$ such that:
\begin{enumerate}[label = (\roman*)]
    \item the natural projection $\PXF^1(\nabla) \rightarrow \mcE$ is surjective; and
    \item $\rm{Ext}^1(\mcE, \CNF(\mcE))=0$.
\end{enumerate}
Then, $\mcE$ is locally free and $c_i(\mcE)=0$ for all $i\ge 1$.
\end{cor}

\begin{cor} Let $\F$ be a foliation on a compact Kahler manifold $X$. Suppose that $\F$ does not have persistent singularities, and either $\rm{Ext}^1(\CNF, \CNF(\CNF))=0$ or $\rm{Ext}^1(N_{\F}, \CNF(N_{\F}))=0$. Then, $\CNF$ is locally free and $c_i(\CNF)=0$ for all $i\ge 1$.
\end{cor}

\subsection{Flat extensions for codimension one foliations}

\begin{lemma}\label{L: flat partial connection factors through F-invariant divisors}
    Let $\F$ be a foliation on a complex smooth variety $X$. Let $(\mcE,\nabla)$ be a partial connection on a reflexive sheaf. Let $D\ge 0$ be a divisor on $X$. Then, $D$ is a $\F$-invariant divisor if, and only if, there exists a partial connection $\nabla'$ on $\mcE(-D)$ such that the natural inclusion $\mcE(-D) \subset \mcE$ is a horizontal morphism.
\end{lemma}
\begin{proof}
    Let us first assume that $D$ is $\F$-invariant. Suppose that $D$ is defined by $f=0$ on an open subset $U \subset X$. The $\F$-invariance of $D$ means that, for all $v\in T_{\F}$, we have that $v(f)/f \in \germe_X$, that is, $df:T_X \rightarrow \germe_X$ induces the morphism $\df(f): T_{\F} \rightarrow \germe_X(-D)$. Then, for every $s\in \mcE$, we have that
    \[
    \nabla(f \cdot s)= \df(f) \otimes s + f \cdot \nabla(s) = \frac{\df(f)}{f} \otimes fs + f \cdot \nabla(s) \in \CTF^1(\mcE(-D)).
    \]
    Therefore, $\nabla$ induces a partial connection on $\mcE(-D)$ such that the inclusion is a horizontal morphism.

    Conversely, assume that there exists a connection $\nabla'$ on $\mcE(-D)$ such that the inclusion is a horizontal morphism. As before, suppose that $D$ is defined by $f=0$ on an open subset $U \subset X$. For every local section $s\in \mcE$ and $v\in T_{\F}$, we have that
    \[
    \nabla_v'(f \cdot s) = \nabla_v(f \cdot s) = v(f) \cdot s + f \cdot \nabla_v(s) \in \mcE(-D),
    \]
    and since $f \cdot \nabla_v(s) \in \mcE(-D)$, it follows that $v(f)\cdot s \in \mcE(-D)$. As this holds for all $s\in \mcE$, we conclude that $v(f) \in \germe_X(-D)$ for all $v\in T_{\F}$. Therefore, $D$ is $\F$-invariant. 
\end{proof}

Let $(\mcE,\nabla)$ be a flat partial connection on a reflexive sheaf. Let $D\ge 0$ be a $\F$-invariant divisor. We define a \emph{horizontal meromorphic splitting} of the short exact sequence associated to $\PXF^1(\nabla)$ as a meromorphic splitting $\sigma: \mcE(-D) \rightarrow \PXF^1(\nabla)$ of Equation (\ref{Eq: short exact sequence of the first sheaf of transverse jets}) that is horizontal with respect to the corresponding flat partial connections: $\nabla^1$ on $\PXF^1(\nabla)$, and the connection given by Lemma \ref{L: flat partial connection factors through F-invariant divisors} on $\mcE(-D)$. Still by Lemma \ref{L: flat partial connection factors through F-invariant divisors}, observe that it does not make sense to talk about a horizontal meromorphic splitting of Equation (\ref{Eq: short exact sequence of the first sheaf of transverse jets}) if $D$ is not $\F$-invariant.

\begin{thm}\label{T: flat meromorphic extensions}
    Let $\F$ be a codimension one foliation on a smooth complex variety $X$. Let $(\mcE,\nabla)$ be a flat partial connection on a reflexive sheaf. Let $D\ge 0$ be a $\F$-invariant divisor on $X$. Then, the correspondence of Theorem \ref{T: meromorphic extensions} induces a correspondence between:
\begin{enumerate}[label = (\alph*)]
    \item\label{Item: flat meromorphic extension} flat meromorphic connections $\hatnabla: \mcE \rightarrow \OmegaDE$ extending $\nabla$; and
    \item\label{Item: horizontal meromorphic splitting} horizontal meromorphic splitting with poles on $D$ of the short exact sequence associated to the first transverse jet of $\nabla$;
\end{enumerate}
\end{thm}

Before proving Theorem \ref{T: flat meromorphic extensions}, we need the following lemma about meromorphic connections.

\begin{lemma}\label{L: codimension one foliation, curvature of meromorphic connection}
    
    Let $\mcE$ be a coherent sheaf on smooth variety $X$, and let $\nabla: \mcE \rightarrow \OmegaDE$ be a meromorphic connection on $\mcE$. Let
    \begin{equation}
        \begin{split}
            \kappa: \mcE & \rightarrow \Omega_X^2(2D)(\mcE) \\
            s & \mapsto \Big(v\wedge w\in \bigwedge^2T_X \mapsto \nabla_{[v,w]}(s) - \nabla_v \circ \nabla_w(s) - \nabla_w \circ \nabla_v(s) \Big)
        \end{split}
    \end{equation}
    be the \emph{curvature} of $\nabla$. Suppose there exists a codimension one foliation $\F$ such that $\kappa(v \wedge w)=0$ for all $v\in T_{\F}$, and $w\in T_X$. Then $\nabla$ is flat.
\end{lemma}
\begin{proof} The proof is similar to that of Proposition \ref{P: codimension one, ideal of persistent singularities}. Let us denote $\mcE' = \mcE(2D)$. Starting with the short exact of the tangent sheaf of $\F$, we naturally deduce the short exact sequence
    \[
    0 \rightarrow \mc K \rightarrow \mcHom(\mcE,\Omega_X^2(\mcE')) \xrightarrow{\phi} \mcHom(\mcE, \mcHom(T_{\F} \otimes T_X, \mcE')),
    \]
    where $\mc K = \mcHom(\mcE, \mcHom(\bigwedge^2(T_X/T_{\F}),\mcE'))$. Since $\F$ has codimension one, it follows that $\mc K=0$, that is, $\phi$ is injective. Observe that $\kappa \in \mcHom(\mcE,\Omega_X^2(\mcE'))$, and for every $v\otimes w\in T_{\F} \otimes T_X$, from the construction of $\phi$ we have that
    \[
    \phi(\kappa)(v\otimes w) = \kappa(w\wedge w) = 0.
    \]
    Hence, $\phi(\kappa)=0$, and because $\phi$ is injective, it follows that $\kappa =0$.
\end{proof}

\begin{proof}[Proof of Theorem \ref{T: flat meromorphic extensions}]
    First, let us prove that a horizontal meromorphic splitting induces a flat meromorphic connection. Let $\sigma: \mcE(-D) \rightarrow \PXF^1(\nabla)$ be the meromorphic splitting, and let $\hatnabla$ the respective meromorphic connection. Since $\F$ has codimension one, by Lemma \ref{L: codimension one foliation, curvature of meromorphic connection}, it is enough to verify that
    \[
    \hatnabla_{[v,w]} - \hatnabla_{v} \circ \hatnabla_w + \hatnabla_w \circ \hatnabla_v =0, \forall v\in T_{\F}, w\in T_X.
    \]
    
    Let $s\in \mcE(-D)$, and let us write
    \[
    \sigma(s) = 1\otimes s - \sum \eta_i \otimes s_i \in \PXF^1(\nabla), \eta_i\in \Omega_X^1, s_i\in \mcE.
    \]
    From this expression, it follows that
    \[
    \hatnabla(s) = \sum \eta_i \otimes s_i.
    \]
    Moreover, since $\sigma$ is flat, for every $v\in T_{\F}$,
    \begin{equation*}
        \begin{split}
            \sigma( \hatnabla_v(s)) & = \sigma( \nabla_v(s)) = \nabla_v^1(\sigma(s)) = \nabla^1_v\Big(1\otimes s - \sum \eta_i \otimes s_i\Big) \\
            & = 1\otimes \nabla_v(s) - \sum \Lie_v(\eta_i)\otimes s_i - \eta_i \otimes \nabla_v(s_i),
        \end{split}
    \end{equation*}
    where we are using the description of $\nabla^1$ given by Equation (\ref{Eq: description of the partial connection on the first sheaf of transverse jets}). Observe, therefore, that we have
    \[
    \hatnabla(\hatnabla_vs) = \sum \Lie_v(\eta_i) \otimes s_i + \eta_i \otimes \nabla_v(s_i).
    \]
    From the equations above, we calculate that
    \begin{equation*}
        \begin{split}
         \kappa_{\hatnabla}(v\wedge w)(s) & =   \hatnabla_{[v,w]}(s) - \hatnabla_{v} \circ \hatnabla_w(s) + \hatnabla_w \circ \hatnabla_v (s) \\
         & = \sum \eta_i([v,w]) \cdot s_i - \hatnabla_v\Big(\sum \eta_i(w)s_i\Big) \\
         & + \Big( \sum \Lie_v(\eta_i)(w)\cdot s_i + \eta_i(w) \cdot \nabla_v(s_i) \Big) \\
         & = \sum \eta_i([v,w]) \cdot s_i - v(\eta_i(w))\cdot s_i + d\eta_i(v,w) +w(\eta_i(w))\cdot s_i =0.
        \end{split}
    \end{equation*}
    Therefore, $\hatnabla$ is flat.

    The converse follows by the same calculations.
\end{proof}

\begin{rmk} Let us compare Theorem \ref{T: flat meromorphic extensions}
with \cite[Theorem 5.4]{fazoli25-arXiv:2505.11662}. Consider the case where $X$ is a complex manifold, $\F$ is a smooth foliation and $\mcE$ is locally free. Let $\sigma: \mcE \rightarrow \PXF^1(\mcE)$ be a horizontal splitting of the short exact sequence associated to the first sheaf of transverse jets of $(\mcE,\nabla)$. Observe that $\sigma$ induces the \emph{transverse homogeneous differential equation}
\begin{equation*}
    \begin{split}
        \phi: \PXF^1(\nabla) & \rightarrow \CNF(\mcE) \\
        \xi & \mapsto \xi - \sigma(\pi(\xi)), 
    \end{split}
\end{equation*}
which in turn induces the $\C$-linear morphism
\begin{equation*}
    \begin{split}
        E: \ker \nabla & \rightarrow  \CNF(\mcE) \\
        s& \mapsto \phi \circ d^1_{X,\mcE}(s)
    \end{split}
\end{equation*}
Let $\hatnabla$ be the flat extension of $\nabla$ corresponding to $\sigma$ by Theorem \ref{T: flat meromorphic extensions}. Using local coordinates, and considering the natural inclusion $\CNF(\mcE) \subset \Omega_X^1(\mcE)$, we have that $E = \hatnabla$, that is, $\ker E  = \ker \hatnabla$. Thus, with the notation of \cite[Theorem 5.4]{fazoli25-arXiv:2505.11662}, $\hatnabla$ is the flat extension of $\nabla$ associated to the transverse differential equation $\phi$. In this sense, the correspondences established in Theorem \ref{T: flat meromorphic extensions} and \cite[Theorem 5.4]{fazoli25-arXiv:2505.11662} are the same.
\end{rmk}

Finally, let us specify Theorem \ref{T: flat meromorphic extensions} for flat partial connections on line bundles. We start by recalling the following well-known proposition.

\begin{prop}\label{P: flat meromorphic extension of a partial connection on line bundle}
    Let $\F$ be a foliation on a complex manifold $X$, and let $\nabla$ be a flat meromorphic connection on a line bundle $\mc L$, with poles on a divisor $D\ge 0$. If $\nabla$ induces a partial connection, then $D$ is $\F$-invariant. 
\end{prop}
\begin{proof}
    Let $U\subset X$ be an open subset where $\mc L$ is free with basis $s\in \mc L$, and $D\cap U = \{f=0\}$ for some holomorphic function. Since $\nabla$ is meromorphic with poles on $D$, we can express $\nabla(s) = (\Omega/f) \otimes s\in \Omega_X^1(D) \otimes \mc L$ for some holomorphic 1-form $\Omega$, and $(\Omega)_0$ with no irreducible component in common with $D$. Moreover, since $\nabla$ is flat, it follows that $d(\Omega/f)=0$, which is the same as $f\cdot d\Omega - df \wedge \Omega =0$. Applying $v\in T_{\F}$ to this equality, we conclude that
    \begin{equation*}\label{Eq: flat condition for a meromorphic connection on a line bundle}
        f \cdot i_vd\Omega - v(f) \cdot \Omega + \Omega(v)\cdot df  =0, \forall v\in T_{\F}.
    \end{equation*}
    Since $\nabla$ induces a partial connection, using the basis description of $\nabla$, it follows that  $\Omega(v)/f \in \germe_X$ for every $v\in T_{\F}$, or in other words, $\Omega(v) \in \germe_X(-D)$ for every $v\in T_{\F}$. 
    From the equation above, it follows that $v(f) \cdot \Omega \in \germe_X(-D) \otimes \Omega_X^1$, and as $(\Omega)_0$ has no irreducible component in common with $D$, we conclude that $v(f) \in \germe_X(-D)$ for all $v\in T_{\F}$. Therefore, $D$ is $\F$-invariant. 
\end{proof}

From Theorem \ref{T: flat meromorphic extensions} and Proposition \ref{P: flat meromorphic extension of a partial connection on line bundle}, we conclude the following corollary.

\begin{cor}
    Let $\F$ be a codimension one foliation on a smooth complex variety $X$, and let $(\mcE,\nabla)$ be a flat partial connection on a reflexive sheaf. Then, using the correspondence of Theorem \ref{T: meromorphic extensions}, a flat meromorphic connection $\hatnabla: \mcE \rightarrow \OmegaDE$ induces a horizontal meromorphic splitting with poles on $D$ of the short exact sequence associated to the first transverse jet of $\nabla$.
\end{cor}

\section{Codimension one foliations on the projective space with no global foliated 1-forms}\label{S: foliations with no global foliated 1-forms}

This section is dedicated to the proof of Theorem \ref{THM: mendson-pereira conjecture}.

\subsection{Unstable rank two sheaves with flat partial connections}

In order to prove the first part of Theorem \ref{THM: mendson-pereira conjecture}, we first need to discuss flat partial connections on rank two unstable sheaves. Recall a reflexive sheaf $\mcE$ on $\Pj^n$ is \emph{stable} if, for every coherent subsheaf $\mcE_0 \subsetneq \mcE$, we have the inequality
\[
\frac{c_1(\mcE_0)}{\mathrm{rk} (\mcE_0)} <  \frac{c_1(\mcE)}{\mathrm{rk}(\mcE)}.
\]
Otherwise, we say that $\mcE$ is \emph{unstable}. See \cite[Section 3]{hartshorne-80-zbMATH03670590} for the basic properties of stable sheaves. 

\begin{lemma}\label{L: unstable rank 2 sheaf with partial connection}
    Let $\F$ be a foliation on $\Pj^n$ with $H^0(\Pj^n, \CTF^1)=0$, and let $(\mcE,\nabla)$ be a flat partial connection on a rank $2$ reflexive sheaf. Let $\mcL\subset \mcE$ be a saturated invertible subsheaf of $\mcE$. If $2c_1(\mcL) \ge c_1(\mcE)$, then $\nabla$ induces a flat partial connection on $\mcL$, that is, $\mc L$ admits a flat partial connection such that the inclusion is a horizontal morphism.
\end{lemma}
\begin{proof}
    The inclusion $\mcL \subset \mcE$ induces the exact sequence
    \[
        0 \rightarrow \mcL \xrightarrow{\iota} \mcE \xrightarrow{\pi} \mc Q \rightarrow 0, 
    \]
    where $\mc Q$ is the quotient of $\mcE$ by $\mcL$. Since $\mc Q$ is torsion-free, the natural morphism $\mc Q \rightarrow \mc Q^{**}$ is injective. Denoting $\mcL' = \mc Q^{**}$, we thus induce a short exact sequence
    \[
        0 \rightarrow \mcL \rightarrow \mcE \rightarrow \mcL'.
    \]
    Applying $\mcHom(T_{\F}, \cdot )$ to the short exact sequence above, we induce the short exact sequence sequence
    \[
        0 \rightarrow \CTF^1(\mcL) \rightarrow \CTF^1(\mcE) \xrightarrow{\pi} \CTF^1(\mc L').
    \]
    Consider the morphism $\phi: \mcL \rightarrow \CTF^1(\mc L')$ defined as the composition
    of the inclusion $\iota: \mcL \rightarrow \mcE$, the connection $\nabla:\mcE \rightarrow \CTF^1(\mcE)$ and the morphism $\pi: \CTF^1(\mcE)\rightarrow  \CTF^1(\mc L')$. The morphism $\phi$ is obviously $\C$-linear, however it is actually $\germe_{\Pj^n}$-linear: for every $f\in \germe_{\Pj^n}$ and $s\in \mcL$,
    \[
    \phi(f \cdot s) =  \pi \circ \nabla (f \cdot s) = \pi(\df(f) \otimes s + f \cdot \nabla(s) ) = f \cdot \phi (s),
    \]
    because $\df(f) \otimes s  \in \CTF^1(\mcL) = \ker(\pi)$. Hence, $\phi \in \Hom_{\germe_{\Pj^n}}(\mcL, \CTF^1(\mc L'))$, and by the natural isomorphism
    \[
    \Hom_{\germe_{\Pj^n}}(\mcL, \CTF^1(\mcL')) = H^0(\Pj^n,\CTF^1 \otimes (\mcL^* \otimes \mcL')),
    \]
    it defines a element on $H^0(\Pj^n,\CTF^1 \otimes (\mcL^* \otimes \mcL'))$. Calculating that 
    \[
    c_1(\mcL^* \otimes \mcL') = c_1(\mcL^*) + c_1(\mcE) - c_1(\mcL) = c_1(\mcE) - 2c_1(\mcL) \le 0,
    \]
    it follows that $H^0(\CTF^1 \otimes (\mcL^* \otimes \mcL'))=0$. Hence, $\phi$ is the zero morphism, and therefore $\nabla: \mcE \rightarrow \CTF^1(\mcE)$ induces a flat partial connection $\nabla': \mcL \rightarrow \CTF^1(\mcL)$ such that the diagram
    \begin{equation*}
    \begin{tikzpicture}
        \matrix(m)[matrix of math nodes, column sep = 2em, row sep = 2em]
        {
        \mcL & \mcE \\
        \CTF^1(\mcL) & \CTF^1(\mcE) \\
        };
        \path[->]
        (m-1-1) edge (m-1-2) edge node[left]{$\nabla'$} (m-2-1)
        (m-1-2) edge node[left]{$\nabla$} (m-2-2)
        (m-2-1) edge (m-2-2)
        ;
    \end{tikzpicture}
    \end{equation*}
    commutes. This concludes the proof of the lemma.
\end{proof}

We will use the lemma above in the following way: if $(\mcE,\nabla)$ is a rank $2$ reflexive sheaf that is unstable in the category of sheaves, then it also unstable in the category of flat partial connections, that is, there exists $(\mcL, \nabla')\subset (\mcE,\nabla)$ such that $c_1(\mcL) \ge 2 c_1(\mcE)$. 

\begin{rmk}
    An interesting problem it to generalize Lemma \ref{L: unstable rank 2 sheaf with partial connection} under more general conditions. That is, determine conditions such that, whenever a unstable sheaf $\mcE$ is endowed with a partial connection $\nabla$, the pair $(\mcE,\nabla)$ is unstable in the category of partial connection as well.  
\end{rmk}

The next proposition is the first part of Theorem \ref{THM: mendson-pereira conjecture}.

\begin{prop}\label{P: foliation with unstable first jet of the Bott connection and without global foliated 1-forms}
    Let $\F$ be a codimension one foliation on $\Pj^n$, $n\ge 3$. Suppose that $H^0(\Pj^n,\CTF^1)=0$. If $\PPjnF^1(\nabla_B)$ is unstable, then $\F$ is transversely affine.
\end{prop}
\begin{proof}
    Since $\PPjnF^1(\nabla_B)$ is unstable, there exists an saturated invertible subsheaf $\mcL \subset \PPjnF^1(\nabla_B)$ with $2c_1(\mcL) \ge c_1(\PXF^1(\nabla_B))$. By Lemma \ref{L: unstable rank 2 sheaf with partial connection}, the connection $\nabla_B^1$ induces a flat partial connection $\nabla$ on $\mcL$ such that the inclusion is horizontal. 

    Let $\iota: (\mcL,\nabla) \rightarrow (\CNF, \nabla_B)$ be the the morphism induced by the projection $\pi: (\PPjnF^1(\nabla),\nabla_B) \rightarrow (\CNF,\nabla_B)$. We claim that $\iota$ is non-vanishing. Otherwise, the inclusion $\mcL \subset  \PPjnF^1(\nabla_B)$ would factorize by the inclusion  $\CNF^{\otimes 2}\subset \PPjnF^1$, that is, it would induce an inclusion $\mcL\subset \CNF^{\otimes 2}$. However,
    \[
    c_1(\mcL) > \frac{c_1(\PPjnF^1(\nabla_B))}{2} = \frac{3 c_1(\CNF)}{2} > 2c_1(\CNF),
    \]
    where in the last inequality we use that $c_1(\CNF)<0$. This leads to a contradiction, and thus $\psi\not\equiv 0$. 
    
    Hence, $\psi$ defines an isomorphism $\mcL \simeq \CNF(-D)$, with $D\ge 0$ $\F$-invariant (see Lemma \ref{L: flat partial connection factors through F-invariant divisors}), and the inclusion $\mcL \subset \PPjnF^1(\nabla_B)$ induces a horizontal splitting of the short exact sequence of $\PPjnF^1(\nabla_B)$. Finally, by Theorem \ref{T: flat meromorphic extensions}, this induces a flat meromorphic extension of $\CNF$. Therefore, $\F$ is transversely affine.
\end{proof}

\begin{rmk}
    At the end of the proof of the proposition above, we can use \cite[Proposition 12.2]{mendson-pereira-22-arXiv:2207.08957} instead of Theorem \ref{T: flat meromorphic extensions} to guarantee that the foliation is indeed virtually transversely additive.
\end{rmk}

\subsection{Rational foliations}

Recall that a foliation $\F$ on $\Pj^n$ is \emph{rational} if there exists irreducible homogeneous polynomials $F_1,F_2\in \C[x_0,\ldots,x_n]$, $\deg(F_1) = d_1$ and $\deg(F_2)=d_2$, such that:
\begin{enumerate}[label = (\roman*)]
    \item $F_1$ and $F_2$ are in general position, that is,
        \begin{equation*}\label{E: transverse hyperplanes}
            dF_1 \wedge d F_2(p) \neq 0, \forall p\in \{ Z_1=Z_2=0\}; \text{ and } 
        \end{equation*}
    \item $(F_1^{d_2}:F_2^{d_1}): \Pj^n \dashrightarrow \Pj^1$ defines $\F$.
\end{enumerate}

In terms of homogeneous 1-form, the rational foliation defined by $F_1^{d_2}/F_2^{d_1}$ is a foliation of $d_1+d_2-2$ foliation defined by the homogeneous 1-form
\[
\omega = d_1F_1dF_2 - d_2F_2dF_1 \in \Omega_{\Pj^n}^1(d_1+d_2)
\]
Hence, the Bott connection on the conormal $\F$ admits two obvious flat meromorphic extensions: the first one is $\nabla^1: \CNF \rightarrow \Omega_X^1(D_1) \otimes \CNF$ induced by the expression
    \[
    d\omega  = \frac{d_1+d_2}{d_1} \frac{dF_1}{F_1} \wedge \omega,
    \]
and the second one is $\nabla^2: \CNF \rightarrow \Omega_X^1(D_2) \otimes \CNF$ induced by
    \[
    d\omega = \frac{d_1+d_2}{d_2} \frac{dF_2}{F_2} \wedge \omega.
    \]

Let $D_1 = \{F_1=0\}, D_2 = \{F_2=0\} \subset \Pj^n$ be the respective $\F$-invariant hypersurfaces of $\Pj^n$. By Theorem \ref{T: flat meromorphic extensions}, the flat meromorphic extensions $\nabla^i$ of $\nabla_B$ induce horizontal meromorphic splittings $\sigma_i: \CNF(-D_i) \rightarrow \PPjnF^1(\nabla_B)$ of the exact sequence of the first jet of $\nabla_B$ on $\CNF$, given respectively by 
    \[
        \sigma_1(F_1\cdot \omega) = -(d_1+d_2)/d_1 \cdot dF_1\otimes \omega + F_1\otimes \omega \in \PPjnF^1(\nabla),
    \]
and 
    \[
        \sigma_2(F_2\cdot \omega) = -(d_1+d_2)/d_2 \cdot dF_2 \otimes \omega + F_2\otimes \omega \in \PPjnF^1(\nabla).
    \]
Considering these morphisms, we prove the following proposition.

\begin{prop}
    Let $\F$ be a codimension one rational foliation on $\Pj^n$, and let $\nabla_B$ be the Bott connection on $\CNF$. Then, $(\PPjnF^1(\nabla_B), \nabla_B^1)$ is the direct sum of two flat partial connections on line bundles. 
\end{prop}
\begin{proof}
    Following the notation above, let us prove that the horizontal morphism
    \[
    \sigma_1 \oplus \sigma_2 : \CNF(-D_1) \oplus \CNF(-D_2) \rightarrow \PPjnF^1(\nabla_B).
    \]
    is an isomorphism. Recall we have the natural short exact sequence
    \[
        0 \rightarrow \CNF \otimes \CNF \rightarrow \PPjnF^1(\nabla_B) \xrightarrow{\pi} \idealpers \otimes \CNF \rightarrow 0.
    \]
    By \cite[Lemma 2.1]{suwa-83-zbMATH03772606}, we have that $\idealpers(\F) = \ideal{F_1,F_2} \subset \germe_{\Pj^n}$. Thus, for every $\xi \in \PPjnF^1(\nabla)$, there exists unique $a,b\in\germe_{\Pj^n}$ such that $\pi(\xi) = (aF_1 + bF_2)\cdot \omega$. Hence,
    \[
    \pi(\eta - \sigma_1(aF_1\cdot \omega) - \sigma_2(bF_2\cdot \omega) ) =0,
    \]
    and thus $\eta - \sigma_1(aF_1\cdot \omega) - \sigma_2(bF_2\cdot \omega) \in \ker(\pi) = \CNF \otimes \CNF$, that is, there exists $c\in \germe_{\Pj^n}$ such that
    \begin{equation*}
        \begin{split}
            \eta - \sigma_1(aF_1\cdot \omega) - \sigma_2(bF_2\cdot \omega) & = c\cdot \omega \otimes \omega \\
        & = c \cdot (d_1F_1dF_2 - d_2F_2dF_1) \otimes \omega \\
        & = \frac{d_1d_2\cdot c}{d_1+d_2}\cdot \sigma_1(F_1F_2\cdot \omega) - \frac{d_1d_2 \cdot c}{d_1+d_2}\cdot \sigma_2(F_1F_2\cdot \omega).    
        \end{split}
    \end{equation*}
    Hence,
    \[
    \eta = \sigma_1\left( \left(a+\frac{d_1d_2\cdot c \cdot F_2}{d_1+d_2}\right)F_1\cdot \omega\right) + \sigma_2\left( \left(b-\frac{d_1d_2\cdot c \cdot F_1}{d_1+d_2}\right)F_2\cdot \omega\right), 
    \]
    that is, $\eta \in \sigma_1(\CNF(-D_1)) + \sigma_2(\CNF(-D_2))$. Moreover, the calculations above show us that this is the only way we can write $\eta$ in terms of the image of $\sigma_1 \oplus \sigma_2$. Therefore, $\sigma_1\oplus \sigma_2$ is an isomorphism. This concludes the proof.
\end{proof}

When $n\ge 3$, we could indeed guarantee the reciprocal of the proposition above, which corresponds to the second part of Theorem \ref{THM: mendson-pereira conjecture}.

\begin{prop}\label{P: foliation with decomposable first jet of the Bott connection and without global foliated 1-forms}
    Let $\F$ be a codimension one foliation on $\Pj^n$, $n\ge 3$. Suppose $H^0(\Pj^n,\CTF^1)=0$. If $(\PPjnF^1(\nabla_B), \nabla^1)$ is the direct sum of two flat partial connections on line bundles, then $\F$ is a rational foliation.
\end{prop}
\begin{proof}
    Let $(\mcL^1,\nabla^1)$ and $(\mcL^2,\nabla^2)$ be flat partial connections on line bunldes such that $(\PPjnF^1(\nabla_B), (\nabla_B)^1) = (\mcL^1,\nabla^1) \oplus (\mcL^2,\nabla^2)$. As in the proof of Proposition \ref{P: foliation with unstable first jet of the Bott connection and without global foliated 1-forms}, 
    we could identify each $\mcL^i$ as a subsheaf of $\CNF$, and there exists $\F$-invariant effective divisors $D_i$ such that $\mcL^i \simeq \CNF(-D_i)$ and the induced morphism $\CNF(-D_i) \rightarrow \CNF$ is the natural inclusion.  Additionally, through this isomorphism, the natural projection $\pi: \PXF^1(\nabla_B) \rightarrow \CNF$ becomes the natural morphism
    \[
    \begin{split}
        \CNF(-D_1)\oplus \CNF(-D_2) & \rightarrow \CNF \\
        \omega_1 \oplus \omega_2 \mapsto \omega_1+\omega_2.
    \end{split}
    \]
    Considering the short exact sequence associated to $\PXF^1(\nabla_B)$, we have
    \begin{equation*}
        \begin{split}
            \idealpers(\F) \otimes \CNF & = \rm{im}(\pi: \PPjnF^1(\nabla_B) \rightarrow \CNF) \\
            & = \rm{im}(\CNF(-D_1) \oplus \CNF(-D_2) \rightarrow \CNF) \\
            & =\CNF(-D_1) + \CNF(-D_2) \subset \CNF.
        \end{split}
    \end{equation*}
    Since $\codim  \mathrm{Spec}(\idealpers(\F)) \ge 2$, it follows that the supports of $D_1$ and $D_2$ consist of distinct irreducible hypersurfaces. Therefore, by \cite[Proposition 12.3]{mendson-pereira-22-arXiv:2207.08957}, the foliation is rational.
\end{proof}

The preceding discussion allows us to reformulate \cite[Conjecture 12.1]{mendson-pereira-22-arXiv:2207.08957} as follows.

\begin{conj}\label{Conjecture: equivalent to Wodson conjecture}
    Let $\F$ be a codimension one foliation on $\Pj^n$, $n\ge 3$. Let $\nabla_B$ be the Bott connection on the conormal sheaf $\CNF$. If $H^0(\Pj^n,\CTF^1)=0$, then  $\PPjnF^1(\nabla_B)$ is isomorphic to the direct sum of two flat partial connections on line bundles.
\end{conj}

\section{Codimension one foliation on the projective plane with singularities of pencil type} \label{S: Codimension one foliation on the projective plane with certain singularities}

In this section, we aim to prove Theorems \ref{THM: pencil 1} and \ref{THM: pencil 2} about foliations on $\Pj^2$ with singularities of pencil type. The proofs of both theorems are very similar and follow essentially from the following lemma.

\begin{lemma}\label{L: extension of the Bott connection with low degree and certain singularities}
    Let $\F$ be a foliation on $\Pj^2$ with singularities of pencil type, and suppose there exists a meromorphic extension $\hatnabla: \CNF \rightarrow \Omega_X^1(D) \otimes \CNF$ with poles on $D\ge 0$ of the Bott connection. If  $\deg(D) \le c_1(N_{\F})/2$, then $\hatnabla$ is flat. 
\end{lemma}

The proof of Lemma \ref{L: extension of the Bott connection with low degree and certain singularities} can be found in Section \ref{S: proof of the lemma of the extension of the Bott connection with low degree and certain singularities}. First, let us restate the theorems and provide the respective proofs assuming the lemma.

\begin{thm}
    Let $\F$ be a degree $d$ foliation on $\Pj^2$ with singularities of pencil type. If $\F$ admits at most $(d+2)^2/4$ radial singularities, then $\F$ admits an $\F$-invariant curve of degree at most $(d+2)/2$ passing by all the radial singularities.
\end{thm}

\begin{proof} First, let us analyze the ideal of persistent singularities. For $p\in \Pj^2$ a singularity give locally by a closed 1-form, $(\idealpers)_p = \germe_p$, and for $p\in \Pj^2$ a radial singularity, $(\idealpers)_p = \m_p$. Hence, for foliations with singularities of pencil type, the length of the ideal of persistent singularities is exactly the number of radial singularities.

By Proposition \ref{P: codimension one, ideal of persistent singularities}, the short exact sequence associated to the first sheaf of transverse jets of the Bott connection is 
\[
0 \rightarrow \CNF^{\otimes 2} \rightarrow \PPjtwoF^1(\nabla_B) \rightarrow \CNF \otimes \idealpers \rightarrow 0,
\]
which allow us to calculate that $c_1(\PPjtwoF^1(\nabla_B)) = -3(d+2)$ and 
\[
c_2(\PPjtwoF^1(\nabla_B) = 2(d+2)^2 + l(\idealpers)  \le 2(d+2)^2 + \frac{(d+2)^2}{4} = \frac{9(d+2)^2}{4}.
\]
Hence, $4 \cdot c_2(\PPjtwoF^1(\nabla_B)) \le c_1^2(\PPjtwoF^1(\nabla_B))$ and, by Schwarzenberger's Lemma (see \cite[Chapter 2, Lemma 1.2.7]{okonek-2011-zbMATH05913322}), it follows that $\PPjtwoF^1(\nabla_B)$ is unstable. That is, there exists a saturated line bundle $\mcL \subset \PPjtwoF^1(\nabla_B)$ with 
\[
2c_1(\mcL) \ge c_1(\PPjtwoF(\nabla_B)) = -3(d+2).
\]
As in the proof of Proposition \ref{P: foliation with unstable first jet of the Bott connection and without global foliated 1-forms}, $\mcL \simeq \CNF(-D)$ for some divisor $D\ge 0$, and since we have the inclusion $\CNF(-D) \subset \CNF \otimes \idealpers$, the divisor $D$ vanishes on the radial singularities of $\F$. Additionally, since $c_1(\mcL) \ge -3(d+2)/2$, it follows that $\deg(D)  \le (d+2)/2$.

Let $\hatnabla: \CNF \rightarrow \Omega_X(D) \otimes \CNF$ be the meromorphic extension of $\nabla_B$ corresponding to the morphism $\sigma: \CNF(-D) \rightarrow \PXF^1(\nabla_B)$ given by Theorem \ref{T: meromorphic extensions}. By Lemma \ref{L: extension of the Bott connection with low degree and certain singularities},  $\hatnabla$ is flat. Therefore, by Proposition \ref{P: flat meromorphic extension of a partial connection on line bundle}, we conclude that $D$ is $\F$-invariant. 
\end{proof}

\begin{thm}
    Let $\F$ be a foliation of degree $2d-2$ on $\Pj^2$. Suppose that $\F$ is a foliation with singularities of pencil type, with exactly $d^2$ radial singularities, given by the intersection of two curves $C_1,C_2$ of degree $d$. Additionally, suppose that no curve of degree less than $d$ passes by all the radial singularities. Then, the foliation $\F$ is given by the pencil of curves determined by $C_1$ and $C_2$.
\end{thm}

\begin{proof}
    This proof is an easy adaptation of the preceding one. Considering the short exact sequence of the first sheaf of transverse jets of the Bott connection on $\CNF$, we have
    \[
    0 \rightarrow \CNF \otimes \CNF \rightarrow \PPjtwoF^1(\nabla_B) \rightarrow \CNF \otimes \idealpers \rightarrow 0,
    \]
    which using the isomorphism $\CNF \simeq \germe_{\Pj^2}(-2d)$, tensoring by $\germe_{\Pj^2}(3d)$, and defining $\mcE = \germe_{\Pj^2}(3d) \otimes \PPjtwoF^1(\nabla_B)$, induces the short exact sequence
    \[
    0 \rightarrow \germe_{\Pj^2}(-d) \rightarrow \mcE \rightarrow \germe_{\Pj^2}(d) \otimes \idealpers \rightarrow 0
    \]
    Since $h^1(\germe_{\Pj^2}(-d))=h^0(\germe_{\Pj^2}(-d))=0$, we use the long exact sequence to conclude that $H^0(\mcE) \rightarrow  H^0(\germe_{\Pj^2}(d) \otimes \idealpers)$ is an isomorphism. Hence, every curve $C$ of the pencil defined by $C_1$ and $C_2$ induces a section $\germe_X \rightarrow \mcE$, and since $H^0(\mcE(-1))=0$, these sections are all saturated. Following as in the preceding theorem, we conclude that each $C$ in the pencil of curves determined by $C_1$ and $C_2$ is $\F$-invariant. Therefore, $\F$ is the foliation determined by the pencil. 
\end{proof}

\subsection{A short exact sequence for foliation on surfaces}
In order to prove Lemma \ref{L: extension of the Bott connection with low degree and certain singularities}, we will need to consider the exact sequence of the conormal sheaf. Let $\F$ be a codimension one foliation on a complex smooth surface $X$. It is well-know in this case (see \cite[Section 2.1]{brunella-04-zbMATH02150908}) that the we have the short exact sequence
\begin{equation*}
    0 \rightarrow \CNF \rightarrow \Omega_X^1 \xrightarrow{\mathrm{restr}} \CTF^1 \otimes \idealsing \rightarrow 0,
\end{equation*}
which applying $\bigwedge^2$ induces
\begin{equation*}
    \CNF \otimes \Omega_X^1 \xrightarrow{\wedge} \Omega_X^2 \rightarrow \bigwedge^2 \left( \CTF^1 \otimes \idealsing \right) \rightarrow 0
\end{equation*}
Remark that $\wedge: \CNF \otimes \Omega_X^1 \rightarrow \Omega_X^2$ is not injective, and the kernel corresponds to the subsheaf $\CNF \otimes \CNF \subset \CNF \otimes \Omega_X^1$. Therefore, we obtain the short exact sequence
\begin{equation}\label{E: short exact sequence of two forms on the surface}
    0 \rightarrow \CNF \otimes (\CTF^1 \otimes \idealsing) \xrightarrow{\wedge} \Omega_X^2 \rightarrow \bigwedge^2 \left( \CTF^1 \otimes \idealsing \right) \rightarrow 0.
\end{equation}
Let us analyze this short exact sequence locally. Let $(x,y)$ be local coordinates for $X$ around a singularity $p\in X$, and let us suppose that $\F$ is given by $\omega = adx + bdy$. Let $\eta \in \CTF^1$ be the local generator such that $\eta(b\cdot d/dx - a\cdot d/dy)=1$. 
With respect to the local generators $dx\wedge dy$ for $\Omega_X^2$ and $\omega \otimes \eta$ for $\CNF \otimes \CTF^1$, the morphism $\CNF \otimes \CTF^1\otimes \idealsing \rightarrow \Omega_X^2$ turns out to be the inclusion $\ideal{a,b} \rightarrow \germe_X$. Hence, in this case, Equation (\ref{E: short exact sequence of two forms on the surface}) is locally the short exact sequence
\begin{equation}\label{E: short exact sequence of the ideal}
    0 \rightarrow \ideal{a,b} \rightarrow \germe_X \rightarrow i_*\germe_{\mr{Spec}(a,b)} \rightarrow 0.
\end{equation}

\subsection{Proof of Lemma \ref{L: extension of the Bott connection with low degree and certain singularities}}\label{S: proof of the lemma of the extension of the Bott connection with low degree and certain singularities}

We start proving that the curvature 
\[
    \kappa(\hatnabla): \CNF \rightarrow \Omega_{\Pj^2}^2(2D) \otimes \CNF
\]
factor through the morphism $\wedge: \CNF \otimes \CTF^1\otimes \idealsing \rightarrow \Omega_{\Pj^2}^2$ described in Equation (\ref{E: short exact sequence of two forms on the surface}). In order to do this, we analyze $\kappa$ on neighborhoods of the singularities.  

Let us set some notation. Let $\omega$ be a 1-form defining the foliation on a neighborhood of a singularity $p$. Let $\eta \in \Omega^1_{\Pj^2}$ be a holomorphic 1-form, and let $f\in \germe_{\Pj^2}$  defining the divisor $D$ such that  $\hatnabla(\omega) = \eta/f \otimes \omega$. We want to prove that
\[
\kappa(\hatnabla) = d\left( \frac{\eta}{f}\right) = \frac{f \cdot d\eta - df \wedge \eta}{f^2} \in (\CNF \wedge \Omega_{\Pj^2}^1) \otimes \germe_{\Pj^2}(2D).
\]

\begin{claim}
    With notation above, if $\omega$ is closed, then $(f\cdot d\eta - df \wedge \eta) \in \CNF \wedge \Omega_{\Pj^2}^1$.
\end{claim}
\begin{proof}
    Since $\omega$ is closed, we have that $\Lie_v(\omega)=0$ for all $v\in T_{\F}$. Thus, $\eta(v)/f =0$ for all $v\in T_{\F}$, and hence $\eta \in \CNF$, that is, $\eta = g \cdot \omega$ for some $g\in \germe_{\Pj^2}$. We calculate that
    \begin{equation*}
        \begin{split}
            f\cdot d\eta - df \wedge \eta & = f \cdot dg \wedge \omega + fg \cdot d\omega - g \cdot df \wedge \omega \\
            & =  (f\cdot dg -g \cdot df ) \wedge \omega \in (\CNF \wedge \Omega_{\Pj^2}^1).
        \end{split}
    \end{equation*}
    This concludes the proof of the claim.
\end{proof}

Observe that, for the same reason, the curvature belongs to $\CNF \wedge \Omega_{\Pj^2}^1$ on a neighborhood of smooth points $p\in \Pj^2$, since in this case the foliation is also locally described by a closed 1-form.

\begin{claim}
    With notation above, if $p\in \Pj^2$ is a radial singularity, then $(f\cdot d\eta - df \wedge \eta) \in \CNF \wedge \Omega_{\Pj^2}^1$.
\end{claim}

\begin{proof}
    By hypothesis, there exist coordinates $(x,y)$ on a neighborhood of $p$ such that $\F$ is defined by the 1-form $\omega = x\cdot  dy - y\cdot dx$. Also in this system of coordinates, the tangent $T_{\F}$ is generated by $v = x \cdot \de/\de x + y \cdot \de/\de y$. Finally, let us write $\eta = A dx + B dy$.
    
    Since $\hatnabla$ is an extension of the Bott connection, we calculate that
    \begin{equation*}
        \begin{split}
            \frac{\eta(v)}{f} \cdot \omega & = \Lie_v(\omega) = d\omega(v) = 2 dx \wedge dy (v)  \\
            & = 2 \cdot (x dy -y dx) \\
            & = 2 \cdot \omega,
        \end{split}
    \end{equation*}
    and thus 
    \[
    x\cdot A+ y \cdot B = 2 \cdot f.
    \]
    Considering the evaluation on $p=0$, it follows that $f(0)=0$. Moreover, considering the first jet of the equation above, we have that
    \[
    x \cdot A(0) + y \cdot B(0) = 2 \cdot \left(x \cdot \frac{\de f}{\de x}(0) + y \cdot \frac{\de f}{\de y}(0)\right),
    \]
    which imply that $A(0) = 2 \cdot \frac{\de f}{\de x}(0)$ and $B(0) = 2 \cdot \frac{\de f}{\de y}(0)$.

    Finally, remember we want to conclude that $df \wedge \eta - f \cdot d\eta \in \CNF \wedge \Omega_{\Pj^2}^1$. By the exact sequence of Equation (\ref{E: short exact sequence of the ideal}), it is enough to conclude that $(df \wedge \eta - f \cdot d\eta)(p)=0$. From the expressions we derived above, we calculate that
    \begin{equation*}
        \begin{split}
            (df \wedge \eta - f \cdot d\eta)(p) & = \left(\frac{\de f}{\de x} \cdot B - \frac{\de f}{\de y} \cdot A\right)(0) \cdot dx \wedge dy - f(0) \cdot d\eta(0) \\
            & = \left( \frac{\de f}{\de x}(0) \cdot B(0) - \frac{\de f}{\de y}(0) \cdot A(0) \right) \cdot dx \wedge dy \\
            & = 0.
        \end{split}
    \end{equation*}
    This concludes the proof of the claim.
    \end{proof}
    
    By the claims above and from the exact sequence of Equation (\ref{E: short exact sequence of two forms on the surface}), it follows that the curvature of $\hatnabla$ factors through the inclusion $\CNF \otimes (\CTF^1 \otimes \idealsing) \rightarrow \Omega_{\Pj^2}^2$, that is, it induces an $\germe_{\Pj^2}$-morphism $\CNF \rightarrow (\CNF \otimes \CTF^1 \otimes \idealsing) \otimes \germe_{\Pj^2}(2D) \otimes \CNF$ such that the diagram
    \begin{equation*}
        \begin{tikzpicture}
            \matrix(m)[matrix of math nodes, column sep = 2em, row sep = 2em]
            {
            & (\CNF \otimes \CTF^1 \otimes \idealsing) \otimes \germe_{\Pj^2}(2D) \otimes \CNF \\
            \CNF & \Omega_{\Pj^2}^2 \otimes \germe_{\Pj^2}(2D) \otimes \CNF \\
            };
            \path[->]
            (m-2-1) edge (m-1-2) edge node[below]{$\kappa(\hatnabla)$} (m-2-2)
            (m-1-2) edge (m-2-2)
            ;
        \end{tikzpicture}
    \end{equation*}
    commutes. Thus, the curvature induces an element
    \[
    \kappa \in H^0(\Pj^2, \CTF^1 \otimes \idealsing \otimes \CNF(2D))
    \]
    Since $\deg(D)\le (d+2)/2$, we have that $\deg(\CNF(2D)) \le 0$. Moreover, since $h^0(\Pj^2, \Omega^1_{\Pj^2})=h^1(\Pj^2, \CNF) = 0$, it follows from the exact sequence on the conormal sheaf of $\F$ that $h^0(\Pj^2,\CTF^1 \otimes \idealsing)=0$. Hence, $h^0(\Pj^2, \CTF^1 \otimes \idealsing \otimes \CNF(2D))=0$, and therefore $\kappa =0$, that is, $\hatnabla$ is flat. \qed

\providecommand{\bysame}{\leavevmode\hbox to3em{\hrulefill}\thinspace}
\providecommand{\MR}{\relax\ifhmode\unskip\space\fi MR }
\providecommand{\MRhref}[2]{%
  \href{http://www.ams.org/mathscinet-getitem?mr=#1}{#2}
}
\providecommand{\href}[2]{#2}

\end{document}